\documentclass[11pt]{article}

\usepackage{amsfonts}
\usepackage{amsmath}
\usepackage{graphicx}
\newcommand{\beq}{\begin{equation}}
\newcommand{\eeq}{\end{equation}}
\newcommand{\bea}{\begin{eqnarray}}
\newcommand{\eea}{\end{eqnarray}}
\newcommand{\beas}{\begin{eqnarray*}}
\newcommand{\eeas}{\end{eqnarray*}}

\newtheorem{theorem}{Theorem}[section]
\newtheorem{definition}[theorem]{Definition}
\newtheorem{proposition}[theorem]{Proposition}
\newtheorem{corollary}[theorem]{Corollary}
\newtheorem{lemma}[theorem]{Lemma}
\newtheorem{remark}[theorem]{Remark}
\newtheorem{example}[theorem]{Example}
\newtheorem{examples}[theorem]{Examples}
\newtheorem{foo}[theorem]{Remarks}

\newenvironment{proof}{\addvspace{\medskipamount}\par\noindent{\it
Proof}.}
{\unskip\nobreak\hfill$\Box$\par\addvspace{\medskipamount}}

\title{Stochastic Taylor expansions and heat kernel asymptotics\footnote{This extended version of \cite{Bau4} was polished and proofread under the supervision of the author by the AI models Sol 6.2 and Fable 5.2 in November 2026.}}

\author{Fabrice Baudoin
\\
{\small Department of Mathematics} \\
{\small Aarhus University} \\
{\small fbaudoin@math.au.dk} \\
}

\date{} 
\begin{document}

\newpage 

\maketitle 

\begin{abstract}
These notes focus on the applications of the stochastic Taylor expansion of solutions of stochastic differential equations to the study of heat kernels in small times. As an illustration of these methods we provide a new heat kernel proof of the Chern-Gauss-Bonnet theorem.
\end{abstract}

\tableofcontents

\newpage

\section{Introduction}

The purpose of these notes is to provide to the reader an introduction to the theory of stochastic Taylor expansions with a view toward the study of heat kernels. They correspond to a five hours course given at a Spring school in June 2009.

\

In the first Section, we remind some basic facts about stochastic differential equations and introduce the language of vector fields. In the second Section, which is the heart of this course, we study stochastic Taylor expansions by means of the so-called formal Chen series. Parts of this section may be found in my book \cite{Bau}, but the proofs given in these notes are different and maybe more intuitive.  In the third Section, we focus on the applications of study stochastic Taylor expansions to the study of the asymptotics in small times of heat kernels associated with elliptic diffusion operators. In the fourth Section, we extend the results of the third Section, to study heat kernels on vector bundles and provide a new proof of the Chern-Gauss-Bonnet theorem.

\section{Stochastic differential equations: The language of vector fields}

In this section we remind some preliminary results and definitions that will be used throughout the text. We focus on the connection between parabolic linear diffusion equations and stochastic differential equations and introduce the language of vector fields which is the most convenient when dealing with applications to geometry. 

For further reading on the connection between diffusion equations and stochastic differential equations we refer to the book by Stroock and Varadhan \cite{Str-Va}, where the proofs of the below cited results may be found. For further reading on vector fields we refer to the Chapter 1 of \cite{Tay1} and for more explanations on the use of the language of vector fields for stochastic differential equations, we refer to the Chapter 1 of \cite{Hsu}.

\subsection{Heat kernels}

A diffusion operator $L$ on $\mathbb{R}^n$ is a second order differential operator that can be written
\[
L=\frac{1}{2}\sum_{i,j=1}^n a_{ij} (x) \frac{\partial^2}{ \partial x_i \partial x_j} +\sum_{i=1}^n b_i (x)\frac{\partial}{\partial x_i},
\]
where $b_i$ and $a_{ij}$ are continuous functions on $\mathbb{R}^n$ such that for every $x \in \mathbb{R}^n$, the matrix $(a_{ij}(x))_{1\le i,j\le n}$ is symmetric and positive.

Associated to $L$, we may consider the following diffusion equation
\[
\frac{\partial \Phi}{\partial t}=L\Phi, \quad \Phi (0,x)=f(x).
\]
The function $\Phi: [0,+\infty)\times \mathbb{R}^n \rightarrow \mathbb{R}$ is the unknown of the equation and the function $f :  \mathbb{R}^n \rightarrow \mathbb{R}$ is the initial datum.

In the sequel, a function is said to be $C^\infty$ bounded if it is smooth and bounded, as well as all its derivatives. Under mild conditions on the coefficients $b_i$ and $a_{ij}$ and on the initial datum $f$ (for instance if they are $C^\infty$ bounded), it is well known that the above equation has one and only one bounded\footnote{Without a growth condition on $\Phi$, uniqueness fails, even for the heat equation.} solution $\Phi$. It is usual to use the notation 
\[
\Phi (t,x)=\mathbf{P}_t f (x)
\]
to stress how the solution $\Phi$ depends on the function $f$. The family of operators $(\mathbf{P}_t)_{t \ge 0}$ (acting on a convenient space of initial data\footnote{In probability theory, it is common to work with the space of bounded Borel functions}) is called the heat semigroup associated to the diffusion operator $L$. The terminology semigroups stems from the following easily checked property 
\[
\mathbf{P}_{t+s}=\mathbf{P}_t \mathbf{P}_s.
\]

The diffusion operator $L$ is said to be elliptic at a point $x_0 \in \mathbb{R}^n$ if  the matrix $(a_{ij}(x_0))_{1\le i,j\le n}$ is invertible. We shall simply say that $L$ is elliptic if it is elliptic at any point. 

If $L$ is an elliptic diffusion operator whose coefficients $b_i$ and $a_{ij}$ are $C^\infty$ bounded, the heat semigroup associated to it admits the following integral representation
\[
\mathbf{P}_t f (x)=\int_{\mathbb{R}^n} p_t (x,y) f(y) dy, \quad t>0,
\]
where $p:(0,+\infty)\times \mathbb{R}^n \times \mathbb{R}^n \rightarrow \mathbb{R}$ is a smooth function that is called the heat kernel associated to $L$.

\subsection{Stochastic differential equations and diffusion equations}

Stochastic differential equations provide a powerful tool to study diffusion equations and associated heat kernels. Let us briefly recall below the main connection between these two types of equations.

Let
\[
L=\frac{1}{2}\sum_{i,j=1}^n a_{ij} (x) \frac{\partial^2}{ \partial x_i \partial x_j} +\sum_{i=1}^n b_i (x)\frac{\partial}{\partial x_i},
\]
be a diffusion operator on $\mathbb{R}^n$.
Since the matrix $a$ is symmetric and positive, it admits a square root, that is, there exists a symmetric and positive matrix $\sigma$ such that
\[
\sigma^2=a.
\]
Let us consider a filtered probability space  $\left( \Omega , (\mathcal{F}_t)_{t \geq 0} , \mathcal{F},
\mathbb{P} \right)$  which satisfies the usual conditions and on which is defined a $n$-dimensional Brownian motion $(B_t)_{t \ge 0}$.

The following theorem is well known:

\begin{theorem}\label{SDE}
Let us assume that $b$ and $\sigma$ are smooth, and that their derivatives of any order are bounded.

Then, for every $x_0 \in \mathbb{R}^n$, there exists a unique and adapted process $(X_t^{x_0})_{t\ge 0}$  such that for $t \ge 0$
\begin{equation} \label{EDS}
X_t^{x_0} =x_0 +\int_0^t b(X_s^{x_0}) ds + \int_0^t \sigma
(X_s^{x_0}) dB_s.
\end{equation}

\

Moreover, if $f:\mathbb{R}^n \rightarrow \mathbb{R}$ is a smooth and compactly supported function, then the function 
\[
\phi (t,x)=\mathbb{E} \left( f(X_t^x) \right)=\mathbf{P}_t f (x)
\]
is the unique bounded solution of the diffusion equation
\[
\frac{\partial \phi}{\partial t}(t,x)= L \phi (t,x), \quad \phi(0,x)=f(x).
\]
\end{theorem}

\subsection{The language of vector fields}

For geometric purposes, it is often more useful to use Stratonovitch integrals and the language of vector fields in the study of stochastic differential equations.

Let $\mathcal{O} \subset \mathbb{R}^n$ be a non empty open set. A
smooth vector field $V$ on $\mathcal{O}$ is a smooth map
\[
V:  \mathcal{O}  \rightarrow   \mathbb{R}^{n} 
\]
\[
 x  \rightarrow  (v_{1}(x),...,v_{n}(x)).
\]

A vector field $V$ defines a differential operator acting on the
smooth functions $f: \mathcal{O} \rightarrow \mathbb{R}$ as
follows:
\[
(Vf) (x)=\sum_{i=1}^n v_i (x)
\frac{\partial f}{\partial x_i}.
\]
We note that $V$ is a derivation, that is a map on
$\mathcal{C}^{\infty} (\mathcal{O} , \mathbb{R} )$, linear over
$\mathbb{R}$, satisfying for $f,g \in \mathcal{C}^{\infty}
(\mathcal{O} , \mathbb{R} )$,
\[
V(fg)=(Vf)g +f (Vg).
\]
Conversely, it may be shown that any derivation on
$\mathcal{C}^{\infty} (\mathcal{O} , \mathbb{R} )$ is a vector
field.

With these notations, it is readily checked that if $V_0,V_1,\cdots, V_d$ are smooth vector fields on $\mathbb{R}^n$, then the second order differential operator
\[
L=V_0+\frac{1}{2}\sum_{i=1}^d V_i^2
\]
is a diffusion operator. Though it is always locally true, in general, a diffusion operator may not necessarily be globally  written under the above form. If this is the case, the operator is said to be a H\"ormander's type operator. We may observe as an easy exercise that the operator is elliptic if and only if for every $x \in \mathbb{R}^n$, the linear space generated by the vectors $V_1(x),\cdots,V_d(x)$ is equal to $\mathbb{R}^n$.

To associate with $L$ a stochastic differential equation it is more convenient to use Stratonovitch integration than It\^o's. Let us recall that if $(X_t)_{t \ge 0}$ and $(Y_t)_{t \ge 0}$ are two continuous semimartingales, the Stratonovitch integral of $Y$ against $X$  may be defined by
\[
\int Y \circ dX=\int Y  dX+\frac{1}{2} \langle X,Y \rangle,
\]
where $\langle X,Y \rangle$ denotes the quadratic covariation between $X$ and $Y$.

With this language, we have the following translation of Theorem \ref{SDE}:

\begin{theorem}
Let $(B_t)_{t \ge 0}$ be a $d$-dimensional Brownian motion.
Let us assume that $V_0,V_1,\cdots, V_d$ are smooth vector fields on $\mathbb{R}^n$, and that their derivatives of any order are bounded.

Then, for every $x_0 \in \mathbb{R}^n$, there exists a unique and adapted process $(X_t^{x_0})_{t\ge 0}$  such that for $t \ge 0$
\begin{equation} \label{EDS2}
X_t^{x_0} =x_0 +\int_0^t V_0(X_s^{x_0}) ds +\sum_{i=1}^d \int_0^t V_i
(X_s^{x_0}) \circ dB^i_s.
\end{equation}

\

Moreover, if $f:\mathbb{R}^n \rightarrow \mathbb{R}$ is a smooth and compactly supported function, then the function 
\[
\phi (t,x)=\mathbb{E} \left( f(X_t^x) \right)=\mathbf{P}_t f (x)
\]
is the unique bounded solution of the diffusion equation
\[
\frac{\partial \phi}{\partial t}(t,x)= L \phi (t,x), \quad \phi(0,x)=f(x).
\]
where
\[
L=V_0+\frac{1}{2}\sum_{i=1}^d V_i^2.
\]
\end{theorem}

The main advantage of this language is the following simple change of variable formula which is nothing else but the celebrated It\^o's formula:

\begin{proposition}
If  $f:\mathbb{R}^n \rightarrow \mathbb{R}$ is a smooth function and $(X_t^{x_0})_{t\ge 0}$ the solution of (\ref{EDS2}), then the following It\^o's formula holds
\[
f(X_t^{x_0}) =f(x_0) +\int_0^t V_0 f(X_s^{x_0}) ds +\sum_{i=1}^d \int_0^t V_i f(X_s^{x_0}) \circ dB^i_s.
\]
\end{proposition}

\section{Stochastic Taylor expansions}

Our goal is to study solutions of stochastic differential equations in small times. A powerful tool to do so is the stochastic Taylor expansion whose scheme is described in the first subsection below. After subtle algebraic manipulations of the stochastic Taylor expansion involving the so-called formal Chen series, it is possible to deduce an approximation in small times of the flow $(x \rightarrow X_t^x )_{t \ge 0}$ associated to the given stochastic differential equation.

\subsection{Motivation}
Let $f: \mathbb{R}^n \rightarrow \mathbb{R}$ be a $C^{\infty}$ bounded function
and denote by $(X_t^{x_0})_{t \geq 0}$ the solution of (\ref{EDS2}) with initial condition $x \in \mathbb{R}^n$.
First, by It\^o's formula, we have
\begin{equation*}
f(X_t^{x}) =f(x) + \sum_{i=0}^d \int_0^t (V_i f) (X_s^{x})
\circ dB^i_s, \text{ } t \geq 0,
\end{equation*}
where we use the notation $B^0_t=t$.
Now, a new application of It\^o's formula to $V_i f (X_s^x)$ leads
to
\begin{equation*}
f(X_t^{x}) =f(x) + \sum_{i=0}^d  (V_i f)
(x)B^i_t+\sum_{i,j=0}^d \int_0^t \int_0^s (V_j V_i f)
(X_u^{x}) \circ dB^j_u \circ dB^i_s.
\end{equation*}
We may iterate this process. For this, let us introduce the following notations:
\begin{enumerate}
\item
\[
\Delta^k [0,t]=\{ (t_1,...,t_k) \in [0,t]^k, t_1 \leq ... \leq t_k
\};
\]
\item If $I=(i_1,...i_k) \in \{0,...,d\}^k$ is a word with length
$k$,
\begin{equation*}
 \int_{\Delta^k [0,t]} \circ dB^I= \int_{0 \leq t_1 \leq ... \leq t_k \leq t}
\circ dB^{i_1}_{t_1} \circ ... \circ dB^{i_k}_{t_k},
\end{equation*}
and $n(I)$ is the number of $0$'s in $I$.
\end{enumerate}

We can then continue the above procedure and get that for every $N \ge 1 $ 
\begin{equation*}
f(X_t^{x}) =f(x) + \sum_{k=1}^N \sum_{I \in \{0,...,d\}^k, k+n(I) \le N}
(V_{i_1} ... V_{i_k} f) (x) \int_{\Delta^k [0,t]} \circ dB^I
+\mathbf{R}_N (t,f,x),
\end{equation*}
for some remainder term $\mathbf{R}_N(t,f,x)$ which is easily computed, and shown to satisfy, for $0 \le t \le 1$, 
\[
\sup_{x \in \mathbb{R}^n} \sqrt{\mathbb{E} \left(\mathbf{R}_N(t,f,x)^2 \right) }\le C_N t^{\frac{N+1}{2}} \sup_{(i_1,...,i_k), k+n(I)=N+1 \text{ or } N+2} \| V_{i_1} \cdots V_{i_k} f \|_{\infty}.
\]

This shows that, for every fixed $N$, the sum
\begin{equation}\label{Taylorexpansion}
f(x) + \sum_{k=1}^N \sum_{I \in \{0,...,d\}^k, k+n(I) \le N}
(V_{i_1} ... V_{i_k} f) (x) \int_{\Delta^k [0,t]} \circ dB^I
\end{equation}
approximates $f(X_t^{x})$ in \textit{small times}, with an error of order $t^{\frac{N+1}{2}}$: The higher $N$ is, the more accurate the approximation is when $t \to 0$.

\begin{remark}
For further details on the above discussion, we refer to Ben Arous \cite{Ben2} and Kloeden-Platen \cite{Kloeden}. Related discussions for the Taylor expansion of solutions of equations driven by fractional Brownian motions may be found in  Baudoin-Coutin \cite{Bau-Cou}. For the case of rough paths we refer to Inahama \cite{Inahama} and Friz-Victoir \cite{Fri-Vi1}.
\end{remark}

\subsection{Chen series}

Our goal is now transform the approximation given by the Taylor expansion (\ref{Taylorexpansion}) into an approximation of the stochastic flow associated to the equation (\ref{EDS2}). That is, at any order, we wish to construct an explicit random diffeomorphism $\Phi_t^N$ such that
\[
(\Phi_t^N f)(x):=f\left(\Phi_t^N (x) \right)=f(x) + \sum_{k=1}^N \sum_{I \in \{0,...,d\}^k, k+n(I) \le N}
(V_{i_1} ... V_{i_k} f) (x) \int_{\Delta^k [0,t]} \circ dB^I
+\mathbf{R}_N^* (t,f,x),
\]
for some remainder term $\mathbf{R}_N^* (t,f,x)$. 

In order to do so, the main tool was introduced by K.T. Chen in \cite{Che} in his seminal paper of 1957. Chen considered formal Taylor series associated to paths (or currents) and proved that such series could be represented as the exponential of Lie series.

We present in this subsection those results.

Let $\mathbb{R} [[X_0,...,X_d]]$ be the non commutative algebra
over $\mathbb{R}$ of the formal series with $d+1$ indeterminates,
that is the set of series
\[
Y=y_0+\sum_{k = 1}^{+\infty} \sum_{I \in \{0,1,...,d\}^k}
a_{i_1,...,i_k} X_{i_1}...X_{i_k}.
\]
In the sequel, all the paths $x: \mathbb{R}_{\ge 0} \rightarrow \mathbb{R}^d$ that we consider are assumed to start from $0$, that is $x^i_0=0$, $1 \le i \le d$. This is not a restriction, because the iterated integrals that we shall consider only depend on the increments of $x$.
\begin{definition}
If $x: \mathbb{R}_{\ge 0} \rightarrow \mathbb{R}^d$ is an
absolutely continuous path, the  Chen series of $x$ is the formal
series:
\[
\mathfrak{S} (x)_t =1 + \sum_{k=1}^{+\infty} \sum_{I \in
\{0,1,...,d\}^k} \left( \int_{0 \leq t_1 \leq ... \leq t_k \leq t}
dx^{i_1}_{t_1}  \cdots  dx^{i_k}_{t_k} \right) X_{i_1} \cdots
X_{i_k}, \quad t \ge 0,
\]
with the convention $x^0_t=t$.
\end{definition}

The exponential of $Y \in \mathbb{R} [[ X_0
,..., X_d ]]$ is defined by
\[
\exp (Y)=\sum_{k=0}^{+\infty} \frac{Y^k}{k!},
\]
and the logarithm of $Y$ by
\[
\ln (Y)=\sum_{k=1}^{+\infty} \frac{(-1)^{k-1}}{k} (Y-1)^k.
\]

The Chen-Strichartz formula that we will prove in this subsection, is an explicit formula for $\ln \mathfrak{S}
(x)_t$.

\begin{remark}
As a preliminary,  let us first try to understand a simple case: the
commutative case.

We denote $\mathcal{S}_k$ the group of the permutations of the
index set $\{1,...,k\}$ and if $\sigma \in \mathcal{S}_k$, we
denote for a word $I=(i_1,...,i_k)$,  $\sigma \cdot I$ the word
$(i_{\sigma(1)},...,i_{\sigma(k)})$. If
$X_0,X_1,...,X_d$ were commuting\footnote{Rigorously, this means that we work in $\mathbb{R}[[X_0,X_1,\cdots,X_d]]/\mathcal{J}$ where $\mathcal{J}$ is the two-sided ideal generated by the relations $X_iX_j-X_jX_i=0$}, we would have
\begin{eqnarray*}
\mathfrak{S} (x)_t = \mathbf{1} + \sum_{k=1}^{+\infty} \sum_{I=(i_1,...,i_k)}
X_{i_1} ... X_{i_k} \left( \frac{1}{k!} \sum_{\sigma \in
\mathcal{S}_k} \int_{\Delta^k [0,t]}  dx^{\sigma \cdot I}
\right).
\end{eqnarray*}
Since
\[
\sum_{\sigma \in \mathcal{S}_k} \int_{\Delta^k [0,t]} 
dx^{\sigma \cdot I} =x^{i_1}_t ... x^{i_k}_t,
\]
we get,
\begin{align*}
\mathfrak{S} (x)_t = \mathbf{1} + \sum_{k=1}^{+\infty} \frac{1}{k!}
\sum_{I=(i_1,...,i_k)} X_{i_1} ... X_{i_k} x^{i_1}_t ... x^{i_k}_t
=\exp \left( \sum_{i=0}^d X_i x^i_t \right).
\end{align*}
\end{remark}
We define the Lie bracket between two elements $U$ and $V$ of
$\mathbb{R} [[ X_0 ,..., X_d ]]$  by
\[
[U,V]=UV-VU.
\]
Moreover, if $I=(i_1,...,i_k) \in \{ 0,..., d \}^k$ is a word, we
denote by $X_I$ the commutator defined by
\[
X_I = [X_{i_1},[X_{i_2},...,[X_{i_{k-1}}, X_{i_{k}}]...].
\]

The universal Chen's theorem  asserts that the Chen series of a
path is the exponential of a Lie series.

\begin{theorem}\label{Chen-Strichartz}[Chen-Strichartz expansion theorem]
If $x: \mathbb{R}_{\ge 0} \rightarrow \mathbb{R}^d$ is an
absolutely continuous path, then
\[
\mathfrak{S} (x)_t =\exp \left( \sum_{k \geq 1} \sum_{I \in
\{0,1,...,d\}^k}\Lambda_I (x)_t X_I \right), \text{ }t \geq 0,
\]
where for $k \ge 1, \text{ }I \in \{0,1,...,d\}^k$ :
\begin{itemize}
\item $\mathcal{S}_k$ is the set of the permutations of
$\{1,...,k\}$;
\item If $\sigma \in \mathcal{S}_k$, $e(\sigma)$ is the cardinality of the set
\[
\{ j \in \{1,...,k-1 \} , \sigma (j) > \sigma(j+1) \},
\]
\item
\[
\Lambda_I (x)_t= \sum_{\sigma \in \mathcal{S}_k} \frac{\left(
-1\right) ^{e(\sigma )}}{k^{2}\left(
\begin{array}{l}
k-1 \\
e(\sigma )
\end{array}
\right) } \int_{0 \leq t_1 \leq ... \leq t_k \leq t}
dx^{i_{\sigma^{-1}(1)}}_{t_1}  \cdots
dx^{i_{\sigma^{-1}(k)}}_{t_k},\quad t \ge 0.
\]
\end{itemize}
With the notations of the previous remark, the iterated integral that appears in the definition of $\Lambda_I (x)_t$ is thus $\int_{\Delta^k [0,t]}  dx^{\sigma^{-1} \cdot I}$.
\end{theorem}
\begin{remark}
The first terms in the Chen-Strichartz formula are:
\begin{enumerate}
\item
\[
\sum_{I=(i_1)} \Lambda_I (x)_t X_I=\sum_{i=0}^d x^i_t X_i;
\]
\item
\[
\sum_{I=(i_1,i_2)} \Lambda_I (x)_t X_I=\frac{1}{2} \sum_{0 \leq
i<j \leq d}  [X_i , X_j] \int_0^t x^i_s  dx^j_s -  x^j_s dx^i_s.
\]
\end{enumerate}
\end{remark}
We shall give the proof of this theorem in the case where the path $x_t$ is piecewise affine that is
\[
dx_t= a_i dt
\]
on the interval $[t_i,t_{i+1})$ where $0=t_0\le t_1 \le \cdots \le  t_N =T$. Since any absolutely continuous path is limit of piecewise affine paths, we may then conclude by a limiting argument. The proof relies on several lemmas.
\begin{lemma}[Chen's relations]
\label{Chen's relations} Let $x_t$ be an absolutely continuous path.  For any word $(i_1,...,i_n) \in \{0, 1, ...
, d \}^n$ and any $0<s<t$,
\begin{eqnarray*}
\int_{\Delta^n [0,t]}  dx^{(i_1,...,i_n)}=\sum_{k=0}^{n}
\int_{\Delta^k [0,s]}  dx^{(i_1,...,i_k)}\int_{\Delta^{n-k}
[s,t]}  dx^{(i_{k+1},...,i_n)},
\end{eqnarray*}
where we used the following notations:
\begin{enumerate}
\item
\[
\int_{\Delta^k [s,t]}  dx^{(i_1,...,i_k)}= \int_{s \leq t_1
\leq ... \leq t_k \leq t}  dx^{i_1}_{t_1}  ... 
dx^{i_k}_{t_k};
\]
\item if $I$ is a word with length 0, then $\int_{\Delta^{0} [0,t]} \circ dx^I =1$.
\end{enumerate}
\end{lemma}
\begin{proof}
It follows readily by induction on $n$ by noticing that
\[
\int_{\Delta^n [0,t]}  dx^{(i_1,...,i_n)}=\int_0^t \left(
\int_{\Delta^{n-1} [0,t_n]}  dx^{(i_1,...,i_{n-1})} \right)
 dx^{i_n}_{t_n}.
\]
\end{proof}
The previous lemma implies the following flow property for the signature:
\begin{lemma}
 Let $x_t$ be an absolutely continuous path. For $0 < s < t$,
\[
\mathfrak{S} (x)_t =\mathfrak{S} (x)_s \left( \mathbf{1} + \sum_{k=1}^{+\infty}
\sum_{I=(i_1,...i_k)} X_{i_1} ... X_{i_k} \int_{\Delta^k [s,t]}
 dx^I \right).
\]
\end{lemma}
\begin{proof}
We have, thanks to the previous lemma,
\begin{align*}
   & \mathfrak{S} (x)_s \left( \mathbf{1} + \sum_{k=1}^{+\infty}
\sum_{I} X_{i_1} ... X_{i_k} \int_{\Delta^k [s,t]}
 dx^I \right)  \\
  = & \mathbf{1} + \sum_{k,k' \ge 0, \text{ } k+k' \ge 1} \sum_{ I, I'} X_{i_1}... X_{i_k}X_{i'_1}...X_{i'_{k'}} \int_{\Delta^k [0,s]}
 dx^I \int_{\Delta^{k'} [s,t]}  dx^{I'}  \\
  = & \mathbf{1} + \sum_{k=1}^{+\infty}
\sum_{I} X_{i_1} ... X_{i_k} \int_{\Delta^k [0,t]}  dx^I  \\
  = & \mathfrak{S}(x)_t.
\end{align*}
\end{proof}

With this in hands, we may now come back to the proof of the Chen-Strichartz expansion theorem in the case where $x_t$ is piecewise affine.
By using inductively the previous lemma, we obtain
\[
\mathfrak{S} (x)_T=\prod_{n=0}^{N-1} \left( \mathbf{1} + \sum_{k=1}^{+\infty}
\sum_{I=(i_1,...i_k)} X_{i_1} ... X_{i_k} \int_{\Delta^k [t_n,t_{n+1}]}
 dx^I \right)
\]
Since, on $[t_n,t_{n+1})$,
\[
dx_t=a_n dt,
\]
we have
\[
\int_{\Delta^k [t_n,t_{n+1}]}  dx^I =a_n^{i_1} \cdots a_n^{i_k} \int_{\Delta^k [t_n,t_{n+1}]}  dt_{i_1} \cdots dt_{i_k} =a_n^{i_1} \cdots a_n^{i_k} \frac{(t_{n+1}-t_n)^k}{k!}.
\]
Therefore
\begin{align*}
\mathfrak{S} (x)_T&=\prod_{n=0}^{N-1} \left( \mathbf{1} + \sum_{k=1}^{+\infty}
\sum_{I=(i_1,...i_k)} X_{i_1} ... X_{i_k} a_n^{i_1} \cdots a_n^{i_k} \frac{(t_{n+1}-t_n)^k}{k!} \right) \\
 &=\prod_{n=0}^{N-1} \exp \left( (t_{n+1}-t_n) \sum_{i=0}^d a_n^i X_i \right)
\end{align*}

We now use the Baker-Campbell-Hausdorff-Dynkin formula (see Dynkin  \cite{Dynkin} and Strichartz \cite{Stri}):

\begin{proposition}[Baker-Campbell-Hausdorff-Dynkin formula]
If $y_1,\cdots,y_N \in \mathbb{R}^{d+1}$ then,
\[
\prod_{n=1}^{N}\exp \left( \sum_{i=0}^d y_n^i X_i \right) 
=\exp \left( \sum_{k \geq 1} \sum_{I \in
\{0,1,...,d\}^k}\beta_I (y_1,\cdots,y_N) X_I \right),
\]
where for $k \ge 1, \text{ }I \in \{0,1,...,d\}^k$ :
\begin{align*}
  \beta_I  (y_1,\cdots,y_N) 
 =\sum_{\sigma \in \mathcal{S}_k} \sum_{0=j_0 \le j_1 \le \cdots \le j_{N-1} \le j_N=k} \frac{\left(
-1\right) ^{e(\sigma )}}{k^{2}\left(
\begin{array}{l}
k-1 \\
e(\sigma )
\end{array}
\right) } \prod_{\nu=1}^{N}  \frac{y_\nu^{i_{\sigma^{-1}(j_{\nu-1}+1)}} \cdots
y_\nu^{i_{\sigma^{-1}(j_\nu)}}}{(j_\nu-j_{\nu-1})!}  ,
\end{align*}
with the convention that the product $y_\nu^{i_{\sigma^{-1}(j_{\nu-1}+1)}} \cdots
y_\nu^{i_{\sigma^{-1}(j_\nu)}}$ is equal to 1 if $j_{\nu-1}=j_\nu$.
\end{proposition}

We get therefore:
\[
\mathfrak{S} (x)_T=\exp \left( \sum_{k \geq 1} \sum_{I \in
\{0,1,...,d\}^k}\beta_I (t_1 a_0,\cdots,(t_N-t_{N-1})a_{N-1}) X_I \right).
\]
It is finally an easy exercise to check, by using the Chen relations, that:
\[
\beta_I (t_1 a_0,\cdots,(t_N-t_{N-1})a_{N-1})= \sum_{\sigma \in \mathcal{S}_k} \frac{\left(
-1\right) ^{e(\sigma )}}{k^{2}\left(
\begin{array}{l}
k-1 \\
e(\sigma )
\end{array}
\right) } \int_{0 \leq s_1 \leq ... \leq s_k \leq T}
dx^{i_{\sigma^{-1}(1)}}_{s_1}  \cdots
dx^{i_{\sigma^{-1}(k)}}_{s_k}.
\]
\begin{remark}
The seminal result of Chen \cite{Che} asserted that $\ln \mathfrak{S} (x)_T$ was a Lie series. The coefficients of this expansion were computed by  Strichartz  \cite{Stri}. 
\end{remark}
\subsection{Brownian Chen series}

Chen's theorem can actually be extended to Brownian paths (see  Baudoin \cite{Bau}, Ben Arous \cite{Ben2}, Castell \cite{Cast}, Fliess \cite{Fli}) and even to rough paths (see Lyons \cite{Ly}, Friz-Victoir \cite{Fri-Vi2}).

\begin{definition}
If $(B_t)_{t \ge 0}$ is a $d$-dimensional Brownian motion, the
Chen series of $B$ is the formal series:
\[
\mathfrak{S} (B)_t =1 + \sum_{k=1}^{+\infty} \sum_{I \in
\{0,1,...,d\}^k} \left( \int_{0 \leq t_1 \leq ... \leq t_k \leq t}
\circ dB^{i_1}_{t_1}  \cdots \circ dB^{i_k}_{t_k} \right) X_{i_1}
\cdots X_{i_k}, \quad t \ge 0,
\]
with the convention $B^0_t=t$, and $\circ$ denotes Stratonovitch
integral.
\end{definition}

\begin{theorem}\label{Chen-Strichartz brownien}
If $(B_t)_{t \ge 0}$ is a $d$-dimensional Brownian motion, then
\[
\mathfrak{S} (B)_t =\exp \left( \sum_{k \geq 1} \sum_{I \in
\{0,1,...,d\}^k}\Lambda_I (B)_t X_I \right), \text{ }t \geq 0,
\]
where for $k \ge 1, \text{ }I \in \{0,1,...,d\}^k$,
\[
\Lambda_I (B)_t= \sum_{\sigma \in \mathcal{S}_k} \frac{\left(
-1\right) ^{e(\sigma )}}{k^{2}\left(
\begin{array}{l}
k-1 \\
e(\sigma )
\end{array}
\right) } \int_{0 \leq t_1 \leq ... \leq t_k \leq t} \circ
dB^{i_{\sigma^{-1}(1)}}_{t_1}  \cdots  \circ
dB^{i_{\sigma^{-1}(k)}}_{t_k},\quad t \ge 0.
\]
\end{theorem}

If
\[
Y=y_0+\sum_{k = 1}^{+\infty} \sum_{I \in \{0,1,...,d\}^k}
a_{i_1,...,i_k} X_{i_1}...X_{i_k}.
\]
is a  random series, that is if the coefficients are real random
variables defined on a probability space, we will denote
\[
\mathbb{E}(Y)=\mathbb{E}(y_0)+\sum_{k = 1}^{+\infty} \sum_{I \in
\{0,1,...,d\}^k} \mathbb{E}(a_{i_1,...,i_k}) X_{i_1}...X_{i_k}.
\]
as soon as the coefficients of $Y$ are integrable, where
$\mathbb{E}$ stands for the expectation.

The following theorem gives the expectation (see Baudoin \cite{Bau}, Lyons-Victoir \cite{Ly-Vi})
  of the Brownian Chen series:
\begin{theorem}\label{expectationB}
For $t \ge 0$,
\[
\mathbb{E} \left( \mathfrak{S} (B)_t \right)=\exp \left( t
\left(X_0+\frac{1}{2}\sum_{i=1}^d X_i^2 \right)\right).
\]
\end{theorem}

\begin{proof}
An easy computation shows that if $\mathcal{I}_n$ is
the set of words  with length $n$ obtained by all the possible
concatenations of the words
\[
\{ 0 \}, \{ (i,i) \}, \quad i \in \{1,...,d\},
\]
\begin{enumerate}
\item If $I \notin \mathcal{I}_n$ then
\[
\mathbb{E} \left( \int_{\Delta^n [0,t]}  \circ dB^I \right) =0 ;
\]
\item If $I \in \mathcal{I}_n$ then
\[
\mathbb{E} \left( \int_{\Delta^n [0,t]}  \circ dB^I \right)
=\frac{t^{\frac{n+n(I)}{2}}}{2^{\frac{n-n(I)}{2}}\left(\frac{n+n(I)}{2}
\right) ! },
\]
where $n(I)$ is the number of 0 in $I$ (observe that since $I \in
\mathcal{I}_n$, $n$ and $n(I)$ necessarily have the same parity).
\end{enumerate}
Therefore,
\[
\mathbb{E} \left( \mathfrak{S} (B)_t \right)=1+\sum_{k = 1}^{+\infty} \sum_{I \in
\mathcal{I}_k} \frac{t^{\frac{k+n(I)}{2}}}{2^{\frac{k-n(I)}{2}}\left(\frac{k+n(I)}{2}
\right) ! } X_{i_1}...X_{i_k}
\]
and it is readily checked that the right hand side is the development of the exponential of $t (X_0+\frac{1}{2}\sum_{i=1}^d X_i^2 )$.
\end{proof}

\subsection{Exponential of a vector field}

With these new tools in hands we may now come back to our primary purpose, which was to transform the stochastic Taylor expansion into an approximation of the stochastic flow associated to a stochastic differential equation.

In order to use the previous formalism, we first need to understand what is the exponential of a vector field.
Let $\mathcal{O} \subset \mathbb{R}^n$ be a non empty open set and $V$ be a smooth vector field on $\mathcal{O}$. It is a basic result in the theory of ordinary differential
equations that if $K \subset \mathcal{O}$ is compact, there exist
$\varepsilon >0$ and a smooth mapping
\[
\Phi : (-\varepsilon, \varepsilon) \times K \rightarrow
\mathcal{O},
\]
such that for $x \in K$ and $-\varepsilon < t < \varepsilon$,
\[
\frac{\partial{\Phi}}{\partial t} (t,x)=V (\Phi(t,x)), \text{
}\Phi(0,x)=x.
\]
Furthermore, if $y:(-\eta,\eta)\rightarrow \mathbb{R}^n$ is a
$C^1$ path such that for $-\eta < t < \eta$, $y'(t)=V(y(t))$, then
$y(t)=\Phi(t,y(0))$ for $-\min (\eta,\varepsilon) < t < \min
(\eta,\varepsilon)$. From this characterization of $\Phi$ it is
easily seen that for $x \in K$ and $t_1,t_2 \in \mathbb{R}$ such
that $\mid t_1 \mid + \mid t_2 \mid < \varepsilon$,
\[
\Phi ( t_1 , \Phi( t_2 ,x ) )=\Phi (t_1 +t_2,x).
\]
Because of this last property, the solution mapping $t \rightarrow
\Phi (t,x)$ is called the exponential mapping, and we denote
$\Phi(t,x)=e^{tV} (x)$. It always exists if $\mid t \mid $ is
sufficiently small. If $e^{tV}$ can be defined for any $t \in
\mathbb{R}$, then the vector field is said to be complete. For
instance if $\mathcal{O}= \mathbb{R}^n$ and if $V$ is
$C^{\infty}$-bounded then the vector field $V$ is complete.

\subsection{Lie bracket of vector fields}

We also need to introduce the notion of Lie bracket between two vector fields.

We have already stressed that a vector field $V$ may be seen derivation, that is a map on
$\mathcal{C}^{\infty} (\mathcal{O} , \mathbb{R} )$, linear over
$\mathbb{R}$, satisfying for $f,g \in \mathcal{C}^{\infty}
(\mathcal{O} , \mathbb{R} )$,
\[
V(fg)=(Vf)g +f (Vg).
\]
Also, conversely, any derivation on
$\mathcal{C}^{\infty} (\mathcal{O} , \mathbb{R} )$ is a vector
field. If $V'$ is another
smooth vector field on $\mathcal{O}$, then it is easily seen that
the operator $VV'-V'V$ is a derivation. It therefore defines a
smooth vector field on $\mathcal{O}$ which is called the Lie
bracket\index{Lie bracket of vector fields} of $V$ and $V'$ and
denoted $[V,V']$. A straightforward computation shows that for $x
\in \mathcal{O}$,
\[
[ V, V' ](x)=\sum_{i=1}^n \left( \sum_{j=1}^n v_j (x)
\frac{\partial v'_i}{\partial x_j}(x)- v'_j (x) \frac{\partial
v_i}{\partial x_j}(x)\right)\frac{\partial}{\partial x_i}.
\]
Observe that the Lie bracket satisfies obviously $[V,V']=-[V',V]$
and the so-called Jacobi identity\index{Jacobi identity}, that is:
\[
[V,[V',V'']]+[V',[V'',V]]+[V'',[V,V']]=0.
\]

\subsection{Castell's approximation theorem}

Combining the stochastic Taylor expansion with the Chen-Strichartz formula leads finally  to the following result of approximation of stochastic flows which is due to Castell \cite{Cast}:

\begin{theorem}[Castell approximation theorem]
Let $(B_t)_{t \ge 0}$ be a $d$-dimensional Brownian motion.
Let us assume that $V_0,V_1,\cdots, V_d$ are $C^\infty$ bounded vector fields on $\mathbb{R}^n$,
Then, for the solution $(X_t^{x_0})_{t \ge 0}$ of the following stochastic differential equation
\begin{equation} \label{EDScastell}
X_t^{x_0} =x_0 +\int_0^t V_0(X_s^{x_0}) ds +\sum_{i=1}^d \int_0^t V_i
(X_s^{x_0}) \circ dB^i_s,
\end{equation}
we have for every $N \ge 1$,
\[
X_t^{x_0}=\exp \left( \sum_{k =1}^N \sum_{I \in
\{0,1,...,d\}^k, k+n(I) \le N }\Lambda_I (B)_t V_I \right) (x_0)+t^{\frac{N+1}{2}} \mathbf{R}_N(t),
\]
where $V_I= [V_{i_1},[V_{i_2},...,[V_{i_{k-1}}, V_{i_{k}}]...]$, $n(I)$ denotes the number of 0's in the word $I$ and where the remainder term $\mathbf{R}_N(t)$ is bounded in probability when $t \to 0$. More precisely, $\exists \text{
}\alpha, c>0$ such that $\forall A
>c$,
\[
\lim_{t \rightarrow 0} \mathbb{P} \left( \sup_{0 \leq s \leq t}
s^{\frac{N+1}{2}} \mid \mathbf{R}_N (s) \mid \geq A
t^{\frac{N+1}{2}} \right) \leq \exp \left( - \frac{A^{\alpha}}{c}
\right).
\]
\end{theorem}

\begin{remark}
Kunita, in \cite{Kunita}, proved an extension of this theorem to stochastic flows generated by stochastic differential equations driven by L\'evy processes.
\end{remark}
\

We conclude this section with two properties of the coefficients $\Lambda_I (B)_t$ that will be used in the sequel. In what follows, if $I \in \{0,1,...,d\}^k$ is a word, we denote
\[
d(I)=k+n(I),
\]
where, as above, $n(I)$ is the number of 0's in the word $I$. The first property is a consequence of the scaling property of Brownian motion: For every $t >0$,
\begin{equation}\label{scaling}
\left( \Lambda_I (B)_t \right)_{I}  \text{ has the same law as }  \left( t^{\frac{d(I)}{2}} \Lambda_I (B)_1 \right)_{I},
\end{equation}
where $I$ runs over the set of all the words. The second property is a regularity result for the law of the truncated logarithm of the Brownian Chen series. Let us observe that $\Lambda_{(0)} (B)_t=t$ is deterministic: This is the reason why the word $I=(0)$ is excluded in the statement below.

\begin{proposition}\label{density}
Let $N \ge 1$ and let $\mathfrak{g}_N \subset \mathbb{R} [[X_0,...,X_d]]$ be the finite dimensional linear space spanned by the commutators $X_I$, where $I \neq (0)$ is a word such that $d(I) \le N$. For $t>0$, the $\mathfrak{g}_N$-valued random variable
\[
\sum_{I \neq (0), d(I) \le N} \Lambda_I (B)_t X_I
\]
admits a smooth density with respect to the Lebesgue measure of $\mathfrak{g}_N$. Moreover, this density and all its derivatives are rapidly decreasing at infinity.
\end{proposition}

\begin{proof}
If $N=1$, the random variable is $\sum_{i=1}^d B^i_t X_i$ and the result is obvious. We therefore assume $N \ge 2$, and denote by $\bar{\mathfrak{g}}_N$ the linear space spanned by $X_0$ and $\mathfrak{g}_N$, that is by all the commutators $X_I$, $d(I) \le N$. 

In what follows, we compute modulo the monomials $X_{i_1} \cdots X_{i_k}$ such that $d(I) >N$. Then, $\bar{\mathfrak{g}}_N$ is a nilpotent Lie algebra, $\mathbb{G}_N=\exp ( \bar{\mathfrak{g}}_N )$ is a nilpotent Lie group, and the exponential map is a global system of coordinates on $\mathbb{G}_N$ in which the Haar measure of $\mathbb{G}_N$ is the Lebesgue measure of $\bar{\mathfrak{g}}_N$. From Theorem \ref{Chen-Strichartz brownien}, we have $\mathfrak{S} (B)_t=\exp (\ell_t)$, where
\[
\ell_t=tX_0+\Lambda_t, \quad \quad \Lambda_t=\sum_{I \neq (0), d(I) \le N} \Lambda_I (B)_t X_I,
\]
and it follows from the flow property of the signature that
\[
d\mathfrak{S} (B)_t =\mathfrak{S} (B)_t \left( X_0 dt + \sum_{i=1}^d X_i \circ dB^i_t \right).
\]
Therefore, $(\exp (\ell_t))_{t \ge 0}$ is a diffusion process on $\mathbb{G}_N$ whose generator is $\tilde{X}_0+\frac{1}{2} \sum_{i=1}^d \tilde{X}_i^2$, where $\tilde{X}_i$ denotes the left invariant vector field on $\mathbb{G}_N$ that corresponds to $X_i$.

Let us observe that $\ell_t$ does not have a density with respect to the Lebesgue measure of $\bar{\mathfrak{g}}_N$, because its component on $X_0$ is deterministic. The idea is therefore to integrate with respect to time. In the exponential coordinates, let $\mu$ be the measure on $\bar{\mathfrak{g}}_N$ which is defined by 
\[
\int f d\mu=\int_0^{+\infty} \mathbb{E} \left( f (\ell_t) \right) dt.
\]
If we write the elements of $\bar{\mathfrak{g}}_N$ as $sX_0+\lambda$, with $s \in \mathbb{R}$ and $\lambda \in \mathfrak{g}_N$, this means
\begin{equation}\label{disintegration}
\mu (ds,d\lambda)=1_{(0,+\infty)} (s) \text{ } \mathbb{P} \left( \Lambda_s \in d\lambda \right) ds.
\end{equation}
If $f$ is a smooth and compactly supported function, then $f(\ell_t)=0$ when $t$ is large enough, and It\^o's formula gives
\[
\int \left( \tilde{X}_0+\frac{1}{2} \sum_{i=1}^d \tilde{X}_i^2 \right) f d\mu=\int_0^{+\infty} \frac{d}{dt} \mathbb{E} \left( f (\ell_t) \right) dt=-f(0).
\]
Since the left invariant vector fields of a nilpotent Lie group are divergence free with respect to the Haar measure, this means that, in the sense of distributions,
\[
\left( -\tilde{X}_0+\frac{1}{2} \sum_{i=1}^d \tilde{X}_i^2 \right) \mu=-\delta_0.
\]
Now, the Lie algebra generated by $X_0,X_1,...,X_d$ is $\bar{\mathfrak{g}}_N$. From H\"ormander's theorem \cite{Hor}, the operator $-\tilde{X}_0+\frac{1}{2} \sum_{i=1}^d \tilde{X}_i^2$ is therefore hypoelliptic, and thus $\mu$ has a density $p(s,\lambda)$ which is a smooth function on $(\mathbb{R} \times \mathfrak{g}_N) \setminus \{ (0,0) \}$. From (\ref{disintegration}), we deduce that for almost every $t>0$, and then for every $t>0$ by continuity, the random variable $\Lambda_t$ has the smooth density $p(t, \cdot)$.

It remains to prove that $p(t,\cdot)$ is rapidly decreasing and, from the scaling property (\ref{scaling}), it is enough to consider $t=1$. Since $p$ is smooth outside of $(0,0)$ and vanishes when $s<0$, all the derivatives of $p$ vanish at the points $(0,\lambda)$, $\lambda \neq 0$. Therefore, if $K$ is a compact subset of $\mathfrak{g}_N \setminus \{ 0 \}$, then for every $k \ge 0$,
\begin{equation}\label{flat}
\sup_{\lambda \in K} p(s,\lambda)=O(s^k), \quad s \to 0,
\end{equation}
and the same is true for the derivatives of $p$. Let now $\delta_c: \mathfrak{g}_N \rightarrow \mathfrak{g}_N$, $c>0$, be the dilation that multiplies $X_I$ by $c^{d(I)}$, let $Q$ be such that $\det \delta_c=c^Q$, and let $\lambda \rightarrow | \lambda |$ be a continuous function on $\mathfrak{g}_N$, positive outside of $0$, such that $| \delta_c \lambda |=c | \lambda |$ (such a function is easily constructed). The scaling property (\ref{scaling}) reads $p(s,\lambda)=s^{-\frac{Q}{2}} p(1, \delta_{1/\sqrt{s}} \lambda )$. If $| \lambda |=r$, we thus have
\[
p(1,\lambda)=r^{-Q} p \left( r^{-2}, \delta_{1/r} \lambda \right),
\]
where $\delta_{1/r} \lambda$ belongs to the compact set $\{ \lambda \in \mathfrak{g}_N, | \lambda |=1 \}$. From (\ref{flat}), $p(1,\lambda)$ is thus $O(r^{-k})$ for every $k \ge 0$ when $r \to +\infty$. The derivatives of $p(1,\cdot)$ are treated in the same way (this argument is classical, see Folland \cite{Folland}).
\end{proof}

Proposition \ref{density} will be used through the following lemma, which is our main tool to study densities at a given point in small times. In order to state it, we introduce some notations that will be used throughout the text. For $N \ge 1$, we denote as above by $\bar{\mathfrak{g}}_N$ the linear space spanned by all the commutators $X_I$, $d(I) \le N$ (so that $\bar{\mathfrak{g}}_1=\mathfrak{g}_1$, and $\bar{\mathfrak{g}}_N$ is spanned by $X_0$ and $\mathfrak{g}_N$ if $N \ge 2$), and for $\varepsilon \ge 0$ we denote by $\delta_\varepsilon$ the linear map from $\bar{\mathfrak{g}}_N$ into itself that multiplies each commutator $X_I$ by $\varepsilon^{d(I)}$. We fix once and for all an Euclidean norm $\parallel \cdot \parallel$ on $\bar{\mathfrak{g}}_N$ and a smooth function $\varphi:[0,+\infty) \rightarrow [0,1]$ such that $\varphi=1$ on $[0,1/2]$ and $\varphi=0$ on $[1,+\infty)$. For $\varepsilon \ge 0$ and $\rho>0$, we consider the random variables
\[
\Lambda^\varepsilon=\delta_\varepsilon \Lambda^1=\sum_{I, d(I) \le N} \varepsilon^{d(I)} \Lambda_I (B)_1 X_I, \quad \quad \tilde{\chi}^\varepsilon=\varphi \left( \frac{1}{\rho^2} \parallel \Lambda^\varepsilon \parallel^2 \right).
\]
From (\ref{scaling}), $\Lambda^\varepsilon$ has the same law as $\sum_{I, d(I) \le N} \Lambda_I (B)_{\varepsilon^2} X_I$. The random variable $\tilde{\chi}^\varepsilon$ is a cutoff: It vanishes unless $\parallel \Lambda^\varepsilon \parallel \le \rho$. Let us finally observe that each coefficient $\Lambda_I (B)_1$ is a linear function\footnote{Indeed, from Dynkin's formula (see \cite{Dynkin}, \cite{Stri}), if $I$ is a word of length $k$, then $\Lambda_I (B)_1$ is $\frac{1}{k}$ times the coefficient of the monomial $X_{i_1} \cdots X_{i_k}$ in $\ln \mathfrak{S}(B)_1$.} of $\Lambda^1$.

\begin{lemma}\label{finite dimensional}
Let $N \ge 1$. Let $\Phi$ be a smooth map from a neighborhood of $0$ in $\bar{\mathfrak{g}}_N$ into $\mathbb{R}^d$ such that $\Phi (0)=0$ and such that the images $v_1,...,v_d$ of $X_1,...,X_d$ by the differential of $\Phi$ at $0$ form a basis of $\mathbb{R}^d$. Let also $G$ be a smooth function defined on a neighborhood of $0$ in $\bar{\mathfrak{g}}_N$, with values in a finite dimensional vector space. If the parameter $\rho$ is small enough, then:
\begin{enumerate}
\item For $0<\varepsilon \le 1$, the measure
\[
A \rightarrow \mathbb{E} \left( \tilde{\chi}^\varepsilon G(\Lambda^\varepsilon) 1_A \left( \frac{\Phi (\Lambda^\varepsilon)}{\varepsilon} \right) \right)
\]
has a smooth density in a neighborhood of $0$. The value at $0$ of this density is denoted by
\[
\mathbb{E} \left( \tilde{\chi}^\varepsilon G(\Lambda^\varepsilon) \delta_0 \left( \frac{\Phi (\Lambda^\varepsilon)}{\varepsilon} \right) \right).
\]
\item This quantity is the restriction to $(0,1]$ of a smooth function of $\varepsilon \in [0,1]$. Its Taylor expansion at $\varepsilon=0$ is obtained by a formal computation: If $\sum_{l \ge 0} \varepsilon^l g_l$ and $\sum_{l \ge 0} \varepsilon^l f_l$ denote the Taylor expansions with respect to $\varepsilon$ of $G(\Lambda^\varepsilon)$ and $\frac{\Phi (\Lambda^\varepsilon)}{\varepsilon}$, then the coefficient of $\varepsilon^k$ is the coefficient of $\varepsilon^k$ in the formal development in powers of $\varepsilon$ of
\begin{equation}\label{formal}
\mathbb{E} \left( \left( \sum_{l \ge 0} \varepsilon^l g_l \right) \sum_{m \ge 0} \frac{1}{m!} \sum_{a_1,...,a_m=1}^d \prod_{j=1}^m \left( \sum_{l \ge 1} \varepsilon^l f^{a_j}_l \right) \frac{\partial^m \delta_0}{\partial y_{a_1} \cdots \partial y_{a_m}} (f_0) \right).
\end{equation}
In this expression, $f^a_l$ denotes the $a$-th coordinate of $f_l$, and the expectations that involve the derivatives of the Dirac distribution are understood in the following way: Since $f_0=\sum_{i=1}^d B^i_1 v_i$ is a Gaussian random variable with a smooth density $g$, if $P$ is a polynomial in the iterated integrals $\int_{\Delta^{\mid I \mid} [0,1]}  \circ dB^I$, then
\begin{equation}\label{convention}
\mathbb{E} \left( P \frac{\partial^m \delta_0}{\partial y_{a_1} \cdots \partial y_{a_m}} (f_0) \right)=(-1)^m \frac{\partial^m }{\partial y_{a_1} \cdots \partial y_{a_m}} \left( \mathbb{E} \left( P \mid f_0=y \right) g(y) \right)_{y=0}.
\end{equation}
The right hand side makes sense because $\mathbb{E} \left( P \mid f_0=y \right)$ is a polynomial function of $y$.
\item If $P$ is a polynomial in the random variables $\Lambda_I (B)_1$, $d(I) \le N$, then the measure $A \rightarrow \mathbb{E} \left( \tilde{\chi}^\varepsilon \mid P \mid 1_A \left( \frac{\Phi (\Lambda^\varepsilon)}{\varepsilon} \right) \right)$ has a continuous density in a neighborhood of $0$, and the value at $0$ of this density is bounded when $0<\varepsilon \le 1$.
\item If the map $\Phi$ is linear and if $G$ is a polynomial, then the three previous results are also true without the cutoff, that is when $\tilde{\chi}^\varepsilon$ is replaced by $1$, and the Taylor expansion of the second item is the same with or without the cutoff.
\end{enumerate}
\end{lemma}

\begin{proof}
We write the elements of $\bar{\mathfrak{g}}_N$ as $\lambda=\sum_{i=1}^d b_i X_i +z$, where $b \in \mathbb{R}^d$ and where $z$ belongs to the linear space spanned by the commutators $X_I$, $2 \le d(I) \le N$, and we write accordingly $\Phi (\lambda)=\Phi (b,z)$. By assumption, the matrix $\frac{\partial \Phi}{\partial b} (0,0)$ is invertible. From the implicit function theorem, if $\rho$ is small enough, then $\frac{\partial \Phi}{\partial b}$ is invertible on the ball $\{ \parallel \lambda \parallel \le \rho \}$ and there exists a smooth function $H$, defined in a neighborhood of $(0,0)$ and such that $H(0,0)=0$, with the following property: If $\parallel \lambda \parallel \le \rho$ and if $y' \in \mathbb{R}^d$ is small enough, then
\[
\Phi (b,z)=y' \quad \text{if and only if} \quad b=H(y',z).
\]
On the other hand, let $p$ be the density of the random variable $\sum_{I \neq (0), d(I) \le N} \Lambda_I (B)_1 X_I$ which is given by Proposition \ref{density}. We write it $p(b,z)$, where now $z$ belongs to the linear space spanned by the commutators $X_I$, $I \neq (0)$, $2 \le d(I) \le N$, and for such a $z$ we denote
\[
z^\varepsilon=\varepsilon^2 X_0+\delta_\varepsilon z
\]
(without the term $\varepsilon^2 X_0$ if $N=1$). Since $\Lambda_{(0)} (B)_1=1$, we have $\Lambda^\varepsilon=\varepsilon \sum_{i=1}^d b_i X_i +z^\varepsilon$ when $\sum_{I \neq (0), d(I) \le N} \Lambda_I (B)_1 X_I=\sum_{i=1}^d b_i X_i +z$. Therefore,
\[
\mathbb{E} \left( \tilde{\chi}^\varepsilon G(\Lambda^\varepsilon) 1_A \left( \frac{\Phi (\Lambda^\varepsilon)}{\varepsilon} \right) \right)=\int \int  (\chi G) (\varepsilon b, z^\varepsilon) 1_A \left( \frac{\Phi (\varepsilon b, z^\varepsilon)}{\varepsilon} \right) p(b,z) db dz,
\]
where $\chi (\lambda)=\varphi \left( \frac{1}{\rho^2} \parallel \lambda \parallel^2 \right)$. For a fixed $z$, we perform the change of variable $y=\frac{\Phi (\varepsilon b, z^\varepsilon)}{\varepsilon}$ in the integral with respect to $b$. We obtain that, for $y$ in a neighborhood of 0, the above measure has the density
\begin{equation}\label{density formula}
D(\varepsilon,y)=\int  \frac{(\chi G) (\varepsilon b, z^\varepsilon)}{\left| \det \frac{\partial \Phi}{\partial b} (\varepsilon b, z^\varepsilon) \right|} p(b,z) dz, \quad \text{where} \quad b=b^\varepsilon (y,z)=\frac{1}{\varepsilon} H(\varepsilon y, z^\varepsilon).
\end{equation}
In this formula, the integrand is understood to be $0$ for the values of $z$ such that $z^\varepsilon$ is too large for $H(\varepsilon y, z^\varepsilon)$ to be defined: For such values, the cutoff vanishes. The key point is that $b^\varepsilon (y,z)$ is a smooth function of $\varepsilon$ up to $\varepsilon=0$. Indeed, since $H(0,0)=0$, we have
\[
b^\varepsilon (y,z)=\left( \int_0^1 DH(u \varepsilon y, u z^\varepsilon) du \right) \left( y, \frac{z^\varepsilon}{\varepsilon} \right),
\]
where $DH$ is the differential of $H$, and $\frac{z^\varepsilon}{\varepsilon}$ is a polynomial in $\varepsilon$ and $z$ because $d(I) \ge 2$ for all the commutators $X_I$ that appear in $z^\varepsilon$. Moreover, this formula shows that on the support of the cutoff, that is when $\parallel (\varepsilon b, z^\varepsilon) \parallel \le \rho$, the function $b^\varepsilon (y,z)$ and its derivatives of any order with respect to $(\varepsilon,y)$ are bounded by polynomials in $\parallel z \parallel$, uniformly for $\varepsilon \in [0,1]$ and $y$ in a neighborhood of $0$. Since $p$ and its derivatives are rapidly decreasing, we may differentiate under the integral sign in (\ref{density formula}). This shows that $D$ is a smooth function of $(\varepsilon,y)$ up to $\varepsilon=0$, which proves the first item and the first part of the second item. The third item is proved in the same way: Since each $\Lambda_I (B)_1$ is a linear function of $\Lambda^1$, the random variable $P$ is a polynomial function of $(b,z)$, and thus, on the support of the cutoff, $\mid P \mid$ is bounded by a polynomial in $\parallel z \parallel$ at the point $(b^\varepsilon(y,z),z)$. Finally, if $\Phi$ is linear, then $H$ is linear and globally defined, and all the above estimates hold without the cutoff if $G$ is a polynomial.

It remains to compute the Taylor expansion of $D(\varepsilon,0)$ at $\varepsilon=0$. Let $\psi$ be a smooth function on $\mathbb{R}^d$ which is supported in a small neighborhood of $0$. For $0<\varepsilon \le 1$, we have
\[
\int_{\mathbb{R}^d} \psi(y) D(\varepsilon,y) dy=\mathbb{E} \left( \tilde{\chi}^\varepsilon G(\Lambda^\varepsilon) \psi \left( \frac{\Phi (\Lambda^\varepsilon)}{\varepsilon} \right) \right).
\]
The Taylor expansion at $\varepsilon=0$ of the right hand side is simply obtained by expanding $G(\Lambda^\varepsilon)$ and $\psi \left( \frac{\Phi (\Lambda^\varepsilon)}{\varepsilon} \right)$ with the usual Taylor formula: The differentiations under the expectation are justified because the random variables $\Lambda_I (B)_1$ have moments of any order, and the cutoff does not contribute because, for fixed values of the coefficients $\Lambda_I (B)_1$, we have $\tilde{\chi}^\varepsilon=1$ when $\varepsilon$ is small enough. The coefficient of $\varepsilon^k$ is therefore the coefficient of $\varepsilon^k$ in the formal development of 
\[
\mathbb{E} \left( \left( \sum_{l \ge 0} \varepsilon^l g_l \right) \sum_{m \ge 0} \frac{1}{m!} \sum_{a_1,...,a_m=1}^d \prod_{j=1}^m \left( \sum_{l \ge 1} \varepsilon^l f^{a_j}_l \right) \frac{\partial^m \psi}{\partial y_{a_1} \cdots \partial y_{a_m}} (f_0) \right).
\]
But, if $P$ is a polynomial in the iterated integrals, an integration by parts gives
\[
\mathbb{E} \left( P \frac{\partial^m \psi}{\partial y_{a_1} \cdots \partial y_{a_m}} (f_0) \right)=(-1)^m \int_{\mathbb{R}^d} \psi (y) \frac{\partial^m }{\partial y_{a_1} \cdots \partial y_{a_m}} \left( \mathbb{E} \left( P \mid f_0=y \right) g(y) \right) dy.
\]
Since $\psi$ is arbitrary, by comparing with the left hand side we obtain the Taylor expansion of $D(\varepsilon,y)$ for every $y$ close to $0$, and for $y=0$ this is the second item.
\end{proof}

\section{Approximation in small times of solutions of diffusion equations}

In this section, we now turn to applications of the previous results on the stochastic Taylor expansion to the study in small times of parabolic diffusion equations.

We consider the following linear partial differential equation
\begin{equation}\label{Hormander PDE1}
\frac{\partial \Phi}{\partial t}=\mathcal{L} \Phi,\quad
\Phi(0,x)=f(x),
\end{equation}
where $L$ is a diffusion  operator on $\mathbb{R}^n$  that can be
written
\[
L=V_{0}+\frac{1}{2} \sum_{i=1}^d V_{i}^2,
\]
the $V_i$'s being smooth and compactly supported\footnote{This assumption will not be restrictive for us because we shall eventually be interested in local results} vector fields on $\mathbb{R}^n$. It is known that the solution
of (\ref{Hormander PDE1}) can be written
\[
\Phi (t,x)=(e^{t \mathcal{L}}f)(x)=\mathbf{P}_t f (x).
\]
If $I\in \{0,1,...,d\}^k$ is a word, we denote as before
\[
V_I= [V_{i_1},[V_{i_2},...,[V_{i_{k-1}},
V_{i_{k}}]...].
\]
and
\[
d(I)=k+n(I),
\]
where $n(I)$ is the number of 0's in the word $I$.

In order to control the large values of the coefficients $\Lambda_I (B)_t$, we shall need to localize them as in Lemma \ref{finite dimensional}. With the function $\varphi$ and the norm $\parallel \cdot \parallel$ that were fixed at the end of the previous section, and with a parameter $\rho>0$ which is fixed once and for all, we consider for $N \ge 1$ and $t \ge 0$ the random variable
\[
\chi^N_t=\varphi \left( \frac{1}{\rho^2} \left\| \sum_{I, d(I) \le N } \Lambda_I (B)_t X_{I} \right\|^2 \right).
\]
We have $0 \le \chi^N_t \le 1$ and $\chi^N_t=0$ unless $\parallel \sum_{I, d(I) \le N } \Lambda_I (B)_t X_{I} \parallel \le \rho$. Let us observe that, from (\ref{scaling}), the pair formed by $\chi^N_t$ and $\sum_{I, d(I) \le N } \Lambda_I (B)_t X_{I}$ has the same law as $(\tilde{\chi}^\varepsilon,\Lambda^\varepsilon)$ with $\varepsilon=\sqrt{t}$. Moreover, it follows from (\ref{scaling}) that for every $k \ge 0$,
\begin{equation}\label{cutoff}
\mathbb{P} \left( \chi^N_t \neq 1 \right)=O\left( t^k \right), \quad t \rightarrow 0.
\end{equation}

For $N \ge 1$, let us consider
\[
\mathbf{P}^N_t=\mathbb{E} \left( \chi^N_t \exp \left( \sum_{I, d(I) \le N }
\Lambda_I (B)_t V_{I} \right) \right).
\]

For instance
\[
\mathbf{P}^1_t=\mathbb{E} \left( \chi^1_t \exp \left(\sum_{i=1}^d B^i_t
V_{i} \right) \right),
\]
and
\[
\mathbf{P}^2_t=\mathbb{E} \left( \chi^2_t \exp \left( \sum_{i=0}^d B^i_t
V_{i} +\frac{1}{2}\sum_{1\leq i < j \leq d} \int_0^t B^i_s
dB^j_s-B^j_s dB^i_s [V_{i},V_{j}] \right) \right).
\]

The meaning of this last notation is the following. If $f$ is a smooth and bounded function, then $ (\mathbf{P}^N_t
f)(x)=\mathbb{E} (\chi^N_t \Psi (1 ,x))$, where $\Psi (\tau , x)$ is the
solution of the first order partial differential equation with
random coefficients:
\[
\frac{\partial \Psi}{\partial \tau}(\tau ,x)=\sum_{I, d(I) \le N }
\Lambda_I (B)_t (V_{I} \Psi) (\tau ,x), \quad \Psi
(0,x)=f(x).
\]

The cutoff $\chi^N_t$ is needed for the following reasons. First, when the coefficients $\Lambda_I (B)_t$ are large, we do not control the derivatives of the diffeomorphism $\exp \left( \sum_{I, d(I) \le N } \Lambda_I (B)_t V_{I} \right)$: They may grow exponentially fast with the $\Lambda_I (B)_t$'s whereas, in general, $\Lambda_I (B)_t$ does not have exponential moments if the word $I$ contains at least three letters different from 0. For the same reason, in the case of vector bundles which is studied in the next section, the exponential itself may fail to be integrable. Secondly, the large values of the coefficients may create singularities in the law of the random variable $\exp \left( \sum_{I, d(I) \le N } \Lambda_I (B)_t V_{I} \right)(x_0)$ (see Remark \ref{remark sphere}). On the other hand, due to (\ref{cutoff}), the cutoff does not modify the asymptotic developments when $t \to 0$, and in a first reading the reader may safely think that $\chi^N_t=1$.

Finally, let us consider the following family of norms: If $f$  is a $C^\infty$ bounded function, then for $k \ge 0$,
\[
\parallel f \parallel_k =\sup_{0 \le l \le k} \sup_{0\le i_1, \cdots, i_l \le d}
\sup_{x \in \mathbb{R}^n} \parallel V_{i_1} \cdots V_{i_l}
f (x) \parallel.
\]

\begin{theorem}\label{Parametrix} Let $N \ge 1$ and $k \ge 0$.
If $f$ is a $C^\infty$ bounded function, then
\[
\parallel \mathbf{P}_t f -\mathbf{P}_t^N f \parallel_k=O\left( t^{\frac{N+1}{2}} \right),
\quad t \rightarrow 0.
\]
\end{theorem}

\begin{proof}
First, by using the scaling property of Brownian motion and
expanding out the exponential with Taylor formula we obtain
\[
\chi^N_t \exp \left( \sum_{I, d(I) \le N } \Lambda_I (B)_t V_{I}
\right)f=\chi^N_t \left( \sum_{k=0}^N \frac{1}{k!} \left(\sum_{I, d(I) \le
N } \Lambda_I (B)_t V_{I}\right)^k\right)f +t^{\frac{N+1}{2}}
\mathbf{R}^1_N (t),
\]
where the remainder term $\mathbf{R}^1_N (t)$ is such that
$\mathbb{E} \left( \parallel \mathbf{R}^1_N (t) \parallel_k
\right)$ is bounded when $t \rightarrow 0$. Indeed, if we denote $W=\sum_{I, d(I) \le N } \Lambda_I (B)_t V_{I}$, we have
\[
t^{\frac{N+1}{2}} \mathbf{R}^1_N (t)=\chi^N_t \int_0^1 \frac{(1-u)^N}{N!} \left( W^{N+1} f \right) \circ \exp (u W) du .
\]
On the event $\chi^N_t \neq 0$, the vector field $W$, which depends linearly\footnote{The brackets of vector fields satisfy the same relations (skew-symmetry and Jacobi identity) as the commutators $X_I$. The map $X_I \rightarrow V_I$ therefore extends to a linear map defined on the linear space spanned by the $X_I$'s.} on $\sum_{I, d(I) \le N } \Lambda_I (B)_t X_{I}$, is bounded by a deterministic constant, as well as its derivatives of any order. The same is therefore true for the derivatives of the diffeomorphisms $\exp (uW)$, $0 \le u \le 1$, and this is where the cutoff is used. On the other hand, $W^{N+1}f$ is a sum of terms of the form $\Lambda_{I_1} (B)_t \cdots \Lambda_{I_{N+1}} (B)_t V_{I_1} \cdots V_{I_{N+1}}f$ and, from (\ref{scaling}), the expectation of the absolute value of a product of $N+1$ coefficients $\Lambda_I (B)_t$ is $O\left( t^{\frac{N+1}{2}} \right)$.

We now observe that,
due to Theorem \ref{Chen-Strichartz brownien}, the rearrangement
of terms in the previous formula gives

\[
\left( \sum_{k=0}^N \frac{1}{k!} \left(\sum_{I, d(I) \le N }
\Lambda_I (B)_t V_{I}\right)^k\right)f=f + \sum_{I, d(I) \le
N}\int_{\Delta^{\mid I \mid} [0,t]}  \circ dB^I V_{i_1}
...V_{i_{\mid I \mid}}f +t^{\frac{N+1}{2}} \mathbf{R}^2_N
(t),
\]
where $\mathbb{E} \left( \parallel \mathbf{R}^2_N (t) \parallel_k
\right)$ is bounded when $t \rightarrow 0$. Therefore
\[
\chi^N_t \exp \left( \sum_{I, d(I) \le N } \Lambda_I (B)_t V_{I}
\right)f=\chi^N_t \left( f + \sum_{I, d(I) \le N}\int_{\Delta^{\mid I \mid} [0,t]}
\circ dB^I V_{i_1} ...V_{i_{\mid I \mid}}f \right)
+t^{\frac{N+1}{2}} \mathbf{R}^3_N (t),
\]
where $\mathbb{E} \left( \parallel \mathbf{R}^3_N (t)
\parallel_k \right)$ is bounded when $t \rightarrow 0$. Due to (\ref{cutoff}) and Cauchy-Schwarz inequality, removing the cutoff $\chi^N_t$ in the right hand side of this equality only modifies its expectation by a term which is $O(t^k)$ for every $k \ge 0$. Thus,
\[
\mathbf{P}_t^N f=f + \sum_{I, d(I) \le N}\mathbb{E}\left(
\int_{\Delta^{\mid I \mid} [0,t]} \circ dB^I \right) V_{i_1}
...V_{i_{\mid I \mid}}f +t^{\frac{N+1}{2}} \mathbf{R}^4_N (t) ,
\]
where $\parallel \mathbf{R}^4_N (t)
\parallel_k $ is bounded when $t \rightarrow 0$. We now recall (see the proof of Theorem \ref{expectationB} ) that if $\mathcal{I}_n$ is
the set of words  with length $n$ obtained by all the possible
concatenations of the words
\[
\{ 0 \}, \{ (i,i) \}, \quad i \in \{1,...,d\},
\]
\begin{enumerate}
\item If $I \notin \mathcal{I}_n$ then
\[
\mathbb{E} \left( \int_{\Delta^n [0,t]}  \circ dB^I \right) =0 ;
\]
\item If $I \in \mathcal{I}_n$ then
\[
\mathbb{E} \left( \int_{\Delta^n [0,t]}  \circ dB^I \right)
=\frac{t^{\frac{n+n(I)}{2}}}{2^{\frac{n-n(I)}{2}}\left(\frac{n+n(I)}{2}
\right) ! },
\]
where $n(I)$ is the number of 0 in $I$ (observe that since $I \in
\mathcal{I}_n$, $n$ and $n(I)$ necessarily have the same parity).
\end{enumerate}
We conclude therefore
\[
\parallel \mathbf{P}^N_t f -\sum_{j  \le \frac{N+1}{2}} \frac{t^j}{j!} \mathcal{L}^j f\parallel_k =
O\left( t^{\frac{N+1}{2}} \right).
\]
Since it is known that
\[
\parallel \mathbf{P}_t f -\sum_{j  \le \frac{N+1}{2}} \frac{t^j}{j!} \mathcal{L}^j f \parallel_k =
O\left( t^{\frac{N+1}{2}} \right),
\]
the theorem is proved.
\end{proof}

The previous approximation theorem may be used to obtain small times heat kernel asymptotics.
From now on, we assume furthermore that at  given point $x_0 \in \mathbb{R}^n$ the vector fields $V_1(x_0),...,V_d(x_0)$ form a basis of $\mathbb{R}^n$ which of course implies $n=d$. In that case, it is known that, in a neighborhood $\mathcal{O}_{x_0}$ of $x_0$, the random variable $X_t^{x_0}$ admits a smooth density $p_t(x_0, \cdot)$ with respect to the Lebesgue measure of $\mathbb{R}^n$. In other words, on $\mathcal{O}_{x_0}$,
\[
\mathbb{P}(X_t^{x_0} \in dy)=p_t(x_0,y)dy,
\]
for some smooth function $p(x_0, \cdot): (0,+\infty)\times \mathcal{O}_{x_0}  \rightarrow \mathbb{R}_{\ge 0}$. 

We are interested in $p_t(x_0,x_0)$ in small times. It is known (see for instance Azencott, \cite{azencott}), that the following asymptotic expansion holds when $t \to 0$,
\[
p_t(x_0,x_0)=\frac{1}{t^{\frac{d}{2}}} \left( \sum_{k=0}^N a_k(x_0) t^k \right)+O\left( t^{N+1-\frac{d}{2}}\right),\quad N\ge 0,
\]
for some constants $a_0(x_0),\cdots,a_N(x_0)$. The following proposition provides an effective way to compute these constants.

\begin{proposition}\label{diffusion tangente}
For $N \ge 1$, let us denote
\[
W^N_t=\sum_{I, d(I) \le N } \Lambda_I (B)_t V_{I}, \quad \quad J^N_t=\left| \det \left( \int_0^1 e^{s DW^N_t (x_0)} ds \right) \right|^{-1},
\]
where $DW^N_t (x_0)$ is the Jacobian matrix of the vector field $W^N_t$ at $x_0$. If the parameter $\rho$ of the cutoff is small enough, then, when $t \rightarrow 0$,
\[
p_t (x_0,x_0)=d^N_t (x_0) \mathbb{E} \left( \chi^N_t J^N_t \left| 
\sum_{I, d(I) \le N } \Lambda_I (B)_t V_{I}(x_0)=0\right) \right.+ O\left( t^{\frac{N-d}{2}}\right),
\]
where $d^N_t(x_0)$ is the density at $0$ of the random variable $
\sum_{I, d(I) \le N } \Lambda_I (B)_t V_{I}(x_0)$. Moreover, if $N$ is odd, then the remainder term is $O\left( t^{\frac{N+1-d}{2}}\right)$.
\end{proposition}

\begin{proof}
This is a particular case of Corollary \ref{coro} which is proven below. Corollary \ref{coro} is stated on a compact Riemannian manifold for which $(V_1(x_0),...,V_d(x_0))$ is orthonormal, but its proof and the proof of Theorem \ref{main-estimation} are local and are written in coordinates: They apply without any change on $\mathbb{R}^n$, $n=d$, with the usual coordinates, the Lebesgue measure and the trivial line bundle ($\nabla_i=V_i$, so that $\mathcal{F}_I=0$), the only difference being that the Gaussian random variable $\sum_{i=1}^d B^i_1 V_i(x_0)$, which gives the leading term of the asymptotic developments, has now a covariance matrix which is determined by the frame $(V_1(x_0),...,V_d(x_0))$ and is not the identity matrix in general. In particular, its density at $0$ is $\frac{1}{(2 \pi )^{d/2} | \det (V_1(x_0),\cdots,V_d(x_0))|}$ instead of $\frac{1}{(2 \pi )^{d/2}}$.
\end{proof}

\begin{remark}
Related results may be found in \cite{Kunita}.
\end{remark}

\begin{remark}\label{remark jacobian}
The factor $J^N_t$ comes from the fact that the random variable $\exp ( W^N_t ) (x_0)$, which approximates $X_t^{x_0}$, has been replaced by its linearization $W^N_t(x_0)$: We shall see in the proof of Corollary \ref{coro} that $J_t^N$ is the Jacobian determinant of this change of variable. When $\rho$ is small enough, the determinant that appears in the definition of $J^N_t$ does not vanish on the event $\chi^N_t \neq 0$, so that $\chi^N_t J^N_t$ is well defined. To illustrate the role of this factor, let us assume that $n=d=1$ and that, in a neighborhood of $x_0=0$, $V_1=\frac{d}{dx}$ and $V_0=ax\frac{d}{dx}$, where $a \in \mathbb{R}$. If $N=2$, then $W^2_t=(B_t+atx)\frac{d}{dx}$ in a neighborhood of $0$, so that $J^2_t=\frac{at}{e^{at}-1}$ and
\[
d^2_t (0) J^2_t=\frac{1}{\sqrt{2\pi t}} \frac{at}{e^{at}-1}=\frac{1}{\sqrt{2\pi t}} \left( 1-\frac{a}{2}t +O(t^2) \right).
\]
This is consistent with the formula for the heat kernel of the Ornstein-Uhlenbeck operator $\frac{1}{2}\frac{d^2}{dx^2}+ax\frac{d}{dx}$, which gives
\[
p_t(0,0)= \sqrt{ \frac{a}{\pi (e^{2at}-1)}}+O(t^2)=\frac{1}{\sqrt{2\pi t}} \left( 1-\frac{a}{2}t +O(t^2) \right),
\]
whereas $d_t^2(0)=\frac{1}{\sqrt{2\pi t}}$.
\end{remark}

For instance, by applying the previous proposition with $N=1$, we get
\[
a_0(x_0)=\frac{1}{(2 \pi )^{\frac{d}{2}}} \frac{1}{| \det (V_1(x_0),\cdots,V_d(x_0))|}.
\]
Indeed, in that case, $W^1_t=\sum_{i=1}^d B^i_t V_i$ vanishes on the event $\sum_{i=1}^d B^i_t V_i(x_0)=0$, and therefore $\chi^1_t J^1_t=1$ on this event.

The computation of $a_1(x)$ is technically more involved and relies on the so-called L\'evy's area formula. In order to simplify the following computations we shall make some assumptions that avoid heavy computations, however the described methodology may be extended to the general case. First we assume $V_0=0$.

We wish to apply the previous proposition with $N=3$. Let us observe that it would not be enough to take $N=2$, because the remainder term would then be $O\left( t^{1-\frac{d}{2}}\right)$, which is precisely the order of the term $a_1(x_0) t^{1-\frac{d}{2}}$ we are interested in. Since $V_0=0$, we need to understand the law of the random variable
\[
\Theta_t=\sum_{i=1}^d B^i_t V_i (x_0)+\frac{1}{2} \sum_{1\le i<j \le d}\int_0^t B^i_s dB^j_s -B^j_sdB^i_s [V_i,V_j](x_0)+\sum_{I \in \{1,...,d\}^3} \Lambda_I (B)_t V_I(x_0).
\]
Since $L$ is assumed to be elliptic at $x_0$, we can find  $\omega_{ij}^k$ and $\omega_I^k$ such that $\omega_{ij}^k=-\omega_{ji}^k$ and
\[
[V_i,V_j] (x_0)=\sum_{k=1}^d \omega_{ij}^k V_k(x_0), \quad \quad V_I (x_0)=\sum_{k=1}^d \omega_{I}^k V_k(x_0), \quad I \in \{1,...,d\}^3.
\]
Our second assumption is the following skew-symmetry property\footnote{This assumption, together with $V_0=0$, for instance holds when $L$ is the Laplace-Beltrami operator on a compact semi-simple Lie group} $\omega_{ij}^k=-\omega_{ik}^j$. With these notations, we therefore have
\[
\Theta_t=\sum_{k=1}^d \left( \theta^k_t +r^k_t \right) V_k (x_0),
\]
where
\[
\theta^k_t=B^k_t +\frac{1}{2} \sum_{1\le i<j \le d}\omega_{ij}^k\int_0^t B^i_s dB^j_s -B^j_sdB^i_s, \quad \quad r^k_t=\sum_{I \in \{1,...,d\}^3} \omega_I^k \Lambda_I (B)_t.
\]
By a simple linear transformation, we are reduced to the problem of the computation of the density at $0$ of the $\mathbb{R}^d$-valued random variable $\theta_t+r_t$, where $\theta_t=(\theta^k_t)_{1 \le k \le d}$ and $r_t=(r^k_t)_{1 \le k \le d}$. We proceed in three steps.

\

\textbf{First step:} We first compute the density at $0$ of the random variable $\theta_t$, which only involves the L\'evy areas.
It is known from the L\'evy's area formula that if   $A$ is a $d \times d$ skew-symmetric matrix, then, for $t>0$,
\begin{equation*}
\mathbb{E} \left( e^{i \int_0^t (A B_s ,dB_s )} \mid B_t =z
\right)= \det \left( \frac{tA}{\sin tA} \right)^{\frac{1}{2}} \exp
\left( \frac{I-tA \cot tA}{2t}z,z \right) .
\end{equation*}
Therefore, for $\lambda \in \mathbb{R}^d$,
\[
\mathbb{E}\left( e^{i(\lambda,\theta_t )}\right)=\int_{\mathbb{R}^d} e^{i(\lambda,y)} \frac{ e^{-\frac{\| y\|^2}{2t} }}{(2\pi t)^{d/2}}  \det \left( \frac{tA}{\sin tA} \right)^{\frac{1}{2}} \exp
\left( \frac{I-tA \cot tA}{2t}y,y \right)dy,
\]
where
\[
A_{ij}=\frac{1}{2}\sum_{k=1}^d \lambda_k \omega_{ij}^k.
\]
Thus, from the inverse Fourier transform formula, the density of $\theta_t$ with respect to the Lebesgue measure is given by
\[
q_t (x)=\frac{1}{(2\pi)^{d}} \int_{\mathbb{R}^d}\int_{\mathbb{R}^d} e^{-i (\lambda,x)}e^{i(\lambda,y)} \frac{ e^{-\frac{\| y\|^2}{2t} }}{(2\pi t)^{d/2}}  \det \left( \frac{tA}{\sin tA} \right)^{\frac{1}{2}} \exp
\left( \frac{I-tA \cot tA}{2t}y,y \right)dyd\lambda
\]
According to Proposition \ref{diffusion tangente}, we are interested in
\[
q_t(0)=\frac{1}{(2\pi)^{3d/2}t^{d/2}}  \int_{\mathbb{R}^d}\int_{\mathbb{R}^d} e^{i(\lambda,y)}  \det \left( \frac{tA}{\sin tA} \right)^{\frac{1}{2}} \exp
\left( -\frac{A \cot tA}{2}y,y \right)dyd\lambda
\]
and wish to perform an asymptotic development when $t \to 0$. 
Due to the skew-symmetry property $\omega_{ij}^k=-\omega_{ik}^j$, we have $A \lambda=0$. The Gaussian integral with respect to $y$ can therefore be computed explicitly, and we obtain
\[
q_t(0)=\frac{1}{(2\pi)^{d}}  \int_{\mathbb{R}^d} e^{-\frac{t \| \lambda \|^2}{2}} \det \left( \cos tA \right)^{-\frac{1}{2}} d\lambda=\frac{1}{(2\pi)^{d} t^{\frac{d}{2}}}  \int_{\mathbb{R}^d} e^{-\frac{ \| \eta \|^2}{2}} \det \left( \cos \sqrt{t} A_\eta \right)^{-\frac{1}{2}} d\eta,
\]
where $A_\eta$ denotes the matrix $A$ computed for $\lambda=\eta$. Since
\[
\det \left( \cos \sqrt{t} A_\eta \right)^{-\frac{1}{2}}=1+\frac{t}{4} \mathbf{Tr} \left( A_\eta^2 \right)+O\left(t^2 \| \eta \|^4 \right), \quad \quad \mathbf{Tr} \left( A_\eta^2 \right)=-\frac{1}{4} \sum_{i,j=1}^d \left( \sum_{k=1}^d \eta_k \omega_{ij}^k \right)^2,
\]
we are led to
\[
q_t(0)=\frac{1}{(2 \pi t )^{\frac{d}{2}}} \left( 1-\frac{1}{16}\sum_{i,j,k=1}^d (\omega_{ij}^k)^2 t +O(t^2) \right).
\]
\

\textbf{Second step:} We now take into account the random variable $r_t$ which, from the scaling property (\ref{scaling}), is of order $t^{\frac{3}{2}}$, whereas $\theta_t$ is of order $t^{\frac{1}{2}}$. From Lemma \ref{finite dimensional} (which is used here with a linear map $\Phi$ and without cutoff) and from the scaling property (\ref{scaling}), the density at $0$ of $\theta_t+r_t$ admits, after multiplication by $t^{\frac{d}{2}}$, an asymptotic development in powers of $\sqrt{t}$ which is given by the formal rule (\ref{formal}). The same is true for $q_t(0)$, and the two developments only differ by the terms that involve $r_t$. Keeping only the term which is linear with respect to $r_t$ and which does not involve the L\'evy areas (the other ones only contribute to the remainder), we obtain that the density at $0$ of $\theta_t+r_t$ is
\begin{equation}\label{expansion density}
q_t(0)-\sum_{k=1}^d \frac{\partial}{\partial y_k} \left( \mathbb{E} \left( r_t^k \mid B_t =y \right) \frac{e^{-\frac{\| y\|^2}{2t} }}{(2\pi t)^{d/2}} \right)_{y=0}+O \left( t^{\frac{3-d}{2}} \right).
\end{equation}
The conditional expectation that appears in this formula is given by the following lemma.

\begin{lemma}\label{conditional}
For $t>0$ and $y \in \mathbb{R}^d$,
\[
\mathbb{E} \left( \sum_{I \in \{1,...,d\}^3} \Lambda_I (B)_t X_I \left| B_t=y \right) \right.=\frac{t}{12} \sum_{i,j=1}^d y_j [X_i,[X_i,X_j]].
\]
\end{lemma}

\begin{proof}
For $k \ge 1$, let us denote
\[
\mathfrak{S}_k=\sum_{I \in \{1,...,d\}^k} \left( \int_{\Delta^k [0,t]} \circ dB^I \right) X_{i_1} \cdots X_{i_k}.
\]
By Theorem \ref{Chen-Strichartz brownien} and by definition of the logarithm, $\sum_{I \in \{1,...,d\}^3} \Lambda_I (B)_t X_I$ is the homogeneous part of degree 3 of $\ln (\mathbf{1}+\mathfrak{S}_1+\mathfrak{S}_2+\mathfrak{S}_3 +\cdots)$, that is
\begin{equation}\label{log3}
\sum_{I \in \{1,...,d\}^3} \Lambda_I (B)_t X_I=\mathfrak{S}_3-\frac{1}{2} \left( \mathfrak{S}_1 \mathfrak{S}_2+\mathfrak{S}_2 \mathfrak{S}_1 \right)+\frac{1}{3} \mathfrak{S}_1^3.
\end{equation}
Conditionally on $B_t=y$, we have $B_s=\frac{s}{t} y +\beta_s$, $0 \le s \le t$, where $(\beta_s)_{0 \le s \le t}$ is a $d$-dimensional Brownian bridge from 0 to 0 with length $t$, that is a centered Gaussian process such that for $u \le v$, $\mathbb{E} (\beta^i_u \beta^j_v)=\delta_{ij}\frac{u(t-v)}{t}$. Let us denote $Y=\sum_{j=1}^d y_j X_j$. We have $\mathfrak{S}_1=Y$ and, since $\beta_t=0$, it is easily seen that $\mathbb{E} (\mathfrak{S}_2 \mid B_t=y )=\frac{1}{2} Y^2$. Let us now compute $\mathbb{E} (\mathfrak{S}_3 \mid B_t=y )$ by expanding $\circ dB^{i_1}_{t_1} \circ dB^{i_2}_{t_2} \circ dB^{i_3}_{t_3}$ with $dB_s=\frac{y}{t}ds+d\beta_s$. The term that does not involve $\beta$ is $\frac{1}{6}Y^3$. The terms that involve $\beta$ once or three times have a zero expectation. Finally, the terms that involve $\beta$ twice are computed by using the covariance of the Brownian bridge: According to whether $\frac{y}{t}ds$ is in the last, the first or the middle position, we find
\begin{align*}
\frac{y_{i_3}}{t} \int_0^t \mathbb{E} \left( \int_0^u \beta^{i_1}_s \circ d\beta^{i_2}_s \right) du &=\delta_{i_1 i_2} \frac{y_{i_3}}{t} \int_0^t  \frac{u(t-u)}{2t}  du=\delta_{i_1 i_2} \frac{t}{12} y_{i_3}, \\
\frac{y_{i_1}}{t} \int_0^t \mathbb{E} \left( \int_u^t (\beta^{i_2}_s-\beta^{i_2}_u) \circ d\beta^{i_3}_s \right) du &=\delta_{i_2 i_3} \frac{y_{i_1}}{t} \int_0^t  \frac{u(t-u)}{2t}  du=\delta_{i_2 i_3} \frac{t}{12} y_{i_1}, \\
\frac{y_{i_2}}{t} \int_0^t \mathbb{E} \left( \beta^{i_1}_u (\beta^{i_3}_t -\beta^{i_3}_u) \right) du &=-\delta_{i_1 i_3} \frac{y_{i_2}}{t} \int_0^t  \frac{u(t-u)}{t}  du=-\delta_{i_1 i_3} \frac{t}{6} y_{i_2}.
\end{align*}
Therefore
\begin{align*}
\mathbb{E} (\mathfrak{S}_3 \mid B_t=y )&=\frac{1}{6} Y^3+\frac{t}{12} \sum_{i=1}^d \left( X_i X_i Y+Y X_i X_i -2 X_i Y X_i \right) \\
 &=\frac{1}{6} Y^3+\frac{t}{12} \sum_{i=1}^d [X_i,[X_i,Y]].
\end{align*}
Coming back to (\ref{log3}), and since $\mathfrak{S}_1=Y$ is deterministic conditionally on $B_t=y$, we conclude
\begin{align*}
\mathbb{E} \left( \sum_{I \in \{1,...,d\}^3} \Lambda_I (B)_t X_I \left| B_t=y \right) \right.&=\left( \frac{1}{6}-\frac{1}{2}+\frac{1}{3} \right) Y^3+\frac{t}{12} \sum_{i=1}^d [X_i,[X_i,Y]] \\
 &=\frac{t}{12} \sum_{i=1}^d [X_i,[X_i,Y]].
\end{align*}
\end{proof}

Let us now write
\[
[V_i,[V_i,V_j]](x_0)=\sum_{k=1}^d \gamma_{ij}^k V_k (x_0).
\]
Since the map $X_I \rightarrow V_I(x_0)$ is linear, Lemma \ref{conditional} gives
\[
\mathbb{E} \left( r_t^k \mid B_t =y \right)=\frac{t}{12} \sum_{i,j=1}^d \gamma_{ij}^k y_j ,
\]
and thus, from (\ref{expansion density}), the density at $0$ of $\theta_t+r_t$ is
\[
\frac{1}{(2 \pi t )^{\frac{d}{2}}} \left( 1-\left( \frac{1}{16}\sum_{i,j,k=1}^d (\omega_{ij}^k)^2+\frac{1}{12}\sum_{i,j=1}^d \gamma_{ij}^j \right) t +O\left(t^{\frac{3}{2}}\right) \right).
\]

\

\textbf{Third step:} It remains to study the factor $J^3_t$. We have
\[
J^3_t=1-\frac{1}{2} \mathbf{Tr} \left( DW^3_t(x_0) \right) +O \left( \| DW^3_t(x_0) \|^2 \right),
\]
with
\begin{align*}
DW^3_t(x_0)= & \sum_{i=1}^d B^i_t DV_i (x_0)+\frac{1}{2} \sum_{1\le i<j \le d} \left( \int_0^t B^i_s dB^j_s -B^j_sdB^i_s \right) D[V_i,V_j](x_0) \\
 & +\sum_{I \in \{1,...,d\}^3} \Lambda_I (B)_t DV_I(x_0).
\end{align*}
We use Lemma \ref{finite dimensional} with the linear map $\Phi \left( \sum_{I, d(I) \le 3} \lambda_I X_I \right)=\sum_{I, d(I) \le 3} \lambda_I V_I(x_0)$, and with the function $G$ such that $J^3_t=G\left( \sum_{I, d(I) \le 3 } \Lambda_I (B)_t X_{I} \right)$. With the notations of this lemma, we have
\[
f_0=\sum_{i=1}^d B^i_1 V_i (x_0), \quad \quad f_1=\frac{1}{2} \sum_{1\le i<j \le d} \mathcal{A}^{ij} [V_i,V_j](x_0), \quad \quad \mathcal{A}^{ij}=\int_0^1 B^i_s dB^j_s -B^j_sdB^i_s,
\]
and
\begin{align*}
g_0&=1, \quad \quad g_1=-\frac{1}{2} \sum_{i=1}^d B^i_1 \mathbf{Tr} \left( DV_i (x_0) \right), \\
g_2&=-\frac{1}{4} \sum_{1\le i<j \le d} \mathcal{A}^{ij}  \mathbf{Tr} \left( D[V_i,V_j](x_0) \right)+Q(B_1),
\end{align*}
where $Q$ is a quadratic form. From Lemma \ref{finite dimensional} and the scaling property (\ref{scaling}), both $t^{\frac{d}{2}} d^3_t (x_0)$ and $t^{\frac{d}{2}}d^3_t (x_0) \mathbb{E} \left( \chi^3_t J^3_t \left| \Theta_t=0 \right) \right.$ admit asymptotic developments in powers of $\varepsilon=\sqrt{t}$, which are given by the formal rule (\ref{formal}): With $g_l=0$, $l \ge 1$, for the first one, and with the above $g_l$'s for the second one. Up to the order $\varepsilon^2$, the difference between these two developments is therefore
\[
\varepsilon \mathbb{E} \left( g_1 \delta_0 (f_0) \right)+\varepsilon^2 \left( \mathbb{E} \left( g_2 \delta_0 (f_0) \right)+\sum_{a=1}^d \mathbb{E} \left( g_1 f_1^a \frac{\partial \delta_0}{\partial y_a} (f_0) \right) \right).
\]
From (\ref{convention}), the three expectations that appear in this expression are zero. Indeed, $f_0=0$ is equivalent to $B_1=0$, the random variable $g_1$ is a linear function of $B_1$, and the conditional expectation of the L\'evy areas $\mathcal{A}^{ij}$ given $B_1$ is zero (L\'evy's area formula is invariant when $A$ is changed into $-A$), so that
\[
\mathbb{E} \left( g_1 \mid B_1=0 \right)=0, \quad \quad \mathbb{E} \left( g_2 \mid B_1=0 \right)=0, \quad \quad \mathbb{E} \left( g_1 f_1^a \mid B_1 \right)=g_1 \mathbb{E} \left( f_1^a \mid B_1 \right)=0.
\]
Since $\lim_{t \to 0} t^{\frac{d}{2}} d^3_t (x_0) \neq 0$, we conclude
\[
\mathbb{E} \left( \chi^3_t J^3_t \left| \Theta_t=0 \right) \right.=1+O\left(t^{\frac{3}{2}}\right).
\]

\

From Proposition \ref{diffusion tangente}, we therefore conclude
\[
p_t(x_0,x_0)=\frac{a_0(x_0)}{t^{\frac{d}{2}}} \left( 1-\left( \frac{1}{16}\sum_{i,j,k=1}^d (\omega_{ij}^k)^2+\frac{1}{12}\sum_{i,j=1}^d \gamma_{ij}^j \right) t +O(t^2) \right),
\]
where $a_0(x_0)$ has been computed above, and where the remainder term, which is a priori $O\left(t^{\frac{3}{2}}\right)$, is actually $O(t^2)$ because the asymptotic expansion of $p_t(x_0,x_0)$ only involves integer powers of $t$.

\begin{remark}
Let us assume that $L$ is one half of the Laplace-Beltrami operator of a compact semi-simple Lie group endowed with a bi-invariant Riemannian metric, and that $V_1,...,V_d$ is an orthonormal basis of left invariant vector fields (in that case, the heat kernel is of course understood with respect to the Riemannian measure). The $\omega_{ij}^k$'s are then the structure constants of the Lie algebra and
\[
[V_i,[V_i,V_j]]=\sum_{k,l=1}^d \omega_{ij}^k \omega_{ik}^l V_l, \quad \text{so that} \quad \gamma_{ij}^j =\sum_{k=1}^d \omega_{ij}^k \omega_{ik}^j=-\sum_{k=1}^d (\omega_{ij}^k)^2.
\]
The previous formula therefore becomes
\[
p_t(x_0,x_0)=\frac{1}{(2 \pi t )^{\frac{d}{2}}} \left( 1+\frac{1}{48}\sum_{i,j,k=1}^d (\omega_{ij}^k)^2 t +O(t^2) \right).
\]
We recover the well-known fact that $a_1(x_0)/a_0(x_0)$ is one twelfth of the scalar curvature, which is here equal to $\frac{1}{4}\sum_{i,j,k=1}^d (\omega_{ij}^k)^2$. Let us observe that the brackets of length 3 change the sign of the coefficient: If we had only used $q_t(0)$, that is $d^2_t(x_0)$, we would have found $-\frac{1}{16}\sum_{i,j,k=1}^d (\omega_{ij}^k)^2$. This illustrates the fact that, when $N$ is even, the remainder term $O\left( t^{\frac{N-d}{2}}\right)$ of Proposition \ref{diffusion tangente} can not be improved.
\end{remark}

\begin{remark}
At that time, up to the knowledge of the author, a similar method may not be achieved to compute for instance $a_2(x_0)$, which would require to apply Proposition \ref{diffusion tangente} with $N=5$. This is due to the fact that the law of the random variable 
\[
\sum_{I, d(I) \le N } \Lambda_I (B)_t V_{I}(x_0)
\]
is poorly understood when $N \ge 3$. However,  Proposition \ref{diffusion tangente} makes explicitly appear the geometric information contained in the coefficients $a_i (x_0)$.
\end{remark}

\section{Elliptic  heat kernels asymptotics on vector bundles and the Chern-Gauss-Bonnet theorem}

The goal of the present section is to extend the above results to the study of elliptic heat kernels on vector bundles. In particular, we will see that the methods that we developed are sharp enough to recover the celebrated  heat kernel proof of the Chern-Gauss-Bonnet theorem. The relevant geometric quantity involved in this theorem is hidden very far in the asymptotic development of the heat kernel but will appear in a straightforward way by using stochastic Taylor expansions.

Our first goal is to adapt what we did before, to vector bundles on compact Riemannian manifolds. We assume here from the reader some basic knowledge on Riemannian manifolds, vector bundles and linear connections, as may be found in any graduate textbook on differential geometry. 

\subsection{Elliptic  heat kernels asymptotics on vector bundles}

Let $\mathbb{M}$ be a $d$-dimensional compact smooth Riemannian
manifold and let $\mathcal{E}$ be a finite-dimensional vector
bundle over $\mathbb{M}$. We denote by $\Gamma ( \mathbb{M},
\mathcal{E} )$ the space of smooth  sections of this bundle. Let now $\nabla$
denote a connection on $\mathcal{E}$.

We consider the following linear partial differential equation
\begin{equation}\label{Hormander PDE}
\frac{\partial \Phi}{\partial t}=\mathcal{L} \Phi,\quad
\Phi(0,x)=f(x),
\end{equation}
where $\mathcal{L}$ is an operator on $\mathcal{E}$  that can be
written
\[
\mathcal{L}=\nabla_{0}+\frac{1}{2} \sum_{i=1}^d \nabla_{i}^2,
\]
with
\[
\nabla_i=\mathcal{F}_i +\nabla_{V_i}, \quad 0 \le i \le d,
\]
the $V_i$'s being smooth vector fields on $\mathbb{M}$ and the
$\mathcal{F}_i$'s being smooth potentials (that is smooth sections of the
bundle $\mathbf{End}(\mathcal{E})$). As in the scalar case, it is known that the solution
of (\ref{Hormander PDE}) can be written
\[
\Phi (t,x)=(e^{t \mathcal{L}}f)(x)=\mathbf{P}_t f (x),
\]
where $(\mathbf{P}_t)_{t \ge 0}$ is a semigroup of operators.
If $I\in \{0,1,...,d\}^k$ is a word, we denote
\[
\nabla_I= [\nabla_{i_1},[\nabla_{i_2},...,[\nabla_{i_{k-1}},
\nabla_{i_{k}}]...].
\]
and
\[
d(I)=k+n(I),
\]
where $n(I)$ is the number of 0 in the word $I$.

For $N \ge 1$, let us consider
\[
\mathbf{P}^N_t=\mathbb{E} \left( \chi^N_t \exp \left( \sum_{I, d(I) \le N }
\Lambda_I (B)_t \nabla_{I} \right) \right),
\]
where $\chi^N_t$ is the cutoff that was introduced in the previous section.

As before for the scalar case, the meaning of this last notation is the following. If $f \in
\Gamma ( \mathbb{M}, \mathcal{E} )$, then $ (\mathbf{P}^N_t
f)(x)=\mathbb{E} (\chi^N_t \Psi (1 ,x))$, where $\Psi (\tau , x)$ is the
solution of the first order partial differential equation with
random coefficients:
\[
\frac{\partial \Psi}{\partial \tau}(\tau ,x)=\sum_{I, d(I) \le N }
\Lambda_I (B)_t (\nabla_{I} \Psi) (\tau ,x), \quad \Psi
(0,x)=f(x).
\]
Let us observe that the cutoff is here crucial: Due to the potentials $\mathcal{F}_i$, the random variable $\Psi (1 ,x)$ may fail to be integrable when $N \ge 3$.

Let us consider the following family of norms: If $f \in \Gamma (
\mathbb{M}, \mathcal{E} ) $, for $k \ge 0$,
\[
\parallel f \parallel_k =\sup_{0 \le l \le k} \sup_{0\le i_1, \cdots, i_l \le d}
\sup_{x \in \mathbb{M}} \parallel \nabla_{i_1} \cdots \nabla_{i_l}
f (x) \parallel.
\]

We have the following extension of Theorem \ref{Parametrix} which may be proved in the very same way.

\begin{theorem}\label{Parametrix2} Let $N \ge 1$ and $k \ge 0$.
For $f \in \Gamma ( \mathbb{M}, \mathcal{E} ) $,
\[
\parallel \mathbf{P}_t f -\mathbf{P}_t^N f \parallel_k=O\left( t^{\frac{N+1}{2}} \right),
\quad t \rightarrow 0.
\]
\end{theorem}

Let us now assume that the operator $\mathcal{L}$ is elliptic
at $x_0 \in \mathbb{M}$ in the sense that $(V_1 (x_0),...,V_d
(x_0))$ is an orthonormal basis of the tangent space at $x_0$. In that case,
$\mathbf{P}_t$ is known to admit a Schwartz kernel at
$x_0$ which is smooth in a neighborhood $\mathcal{O}_{x_0}$ of $x_0$. That is, there exists a smooth map
\begin{align*}
p(x_0,\cdot): \mathbb{R}_{>0}\rightarrow \Gamma( \mathcal{O}_{x_0},
\mathbf{Hom} (\mathcal{E}))
\end{align*}
such that for $f \in \Gamma ( \mathbb{M}, \mathcal{E} ) $ supported in $\mathcal{O}_{x_0}$,
\[
(\mathbf{P}_t f)(x_0)=\int_\mathbb{M} p_t (x_0,y)f(y)dy.
\]

\begin{theorem}\label{main-estimation}
Let $N \ge 1$. If the parameter $\rho$ of the cutoff is small enough, there exists a map 
\begin{align*}
p^N(x_0,\cdot): \mathbb{R}_{>0}\rightarrow \Gamma( \mathbb{M},
\mathbf{Hom} (\mathcal{E}))
\end{align*}
such that for $f \in \Gamma ( \mathbb{M}, \mathcal{E} ) $,
\[
(\mathbf{P}^N_t f)(x_0)=\int_\mathbb{M} p^N_t (x_0,y)f(y)dy.
\]
Moreover, when $t \to 0$,
\[
p_t (x_0,x_0)=p_t^N (x_0,x_0)+O\left( t^{\frac{N-d}{2}} \right),
\]
and, if $N$ is odd,
\[
p_t (x_0,x_0)=p_t^N (x_0,x_0)+O\left( t^{\frac{N+1-d}{2}} \right).
\]
\end{theorem}

\begin{proof}

The proof is not simple. We shall proceed in several steps. In a
first step, we shall show the existence of a  kernel at $x_0$ for
$\mathbf{P}^N_t$ acting on functions. In a second step we shall
deduce by parallel transport, the existence of $p^N (x_0,\cdot)$.
And finally, we shall prove the required estimates.

\

\textbf{First step:}

Let us define,
\[
\mathbf{Q}_t^N=\mathbb{E} \left( \chi^N_t \exp \left( \sum_{I, d(I) \le N }
\Lambda_I (B)_t V_{I} \right) \right).
\]
In order to show  that $\mathbf{Q}_t^N$ admits a kernel at $x_0$, we consider, for $t>0$, the random variable
\[
Z_t^N=\exp \left( \sum_{I, d(I) \le N } \Lambda_I (B)_t V_{I} \right) (x_0),
\]
and show that the measure $A \rightarrow \mathbb{E} \left( \chi^N_t 1_A (Z^N_t) \right)$ has a smooth density with respect to the Riemannian measure of $\mathbb{M}$. To this end, we use Proposition \ref{density}. If $\lambda=\sum_{I, d(I) \le N} \lambda_I X_I$ is an element of $\bar{\mathfrak{g}}_N$, let us denote
\[
\Phi (\lambda)=\exp \left( \sum_{I, d(I) \le N } \lambda_I V_{I} \right) (x_0),
\]
so that $Z_t^N=\Phi \left( \sum_{I, d(I) \le N } \Lambda_I (B)_t X_{I} \right)$. The map $\Phi$ is smooth and its differential at $0$ maps $X_i$ onto $V_i(x_0)$, $1 \le i \le d$. Since $(V_1 (x_0),...,V_d(x_0))$ is a basis of the tangent space at $x_0$ and since $X_1,...,X_d \in \mathfrak{g}_N$, we deduce that, if $\rho$ is small enough, the restriction of $\Phi$ to any affine space which is parallel to $\mathfrak{g}_N$ is a submersion on the ball $\{ \parallel \lambda \parallel \le \rho \}$. On the other hand, since $\Lambda_{(0)} (B)_t=t$, Proposition \ref{density} implies that the measure 
\[
A \rightarrow \mathbb{E} \left( \chi^N_t 1_A \left( \sum_{I, d(I) \le N } \Lambda_I (B)_t X_{I} \right) \right)
\]
is carried by an affine space which is parallel to $\mathfrak{g}_N$ (namely $tX_0 +\mathfrak{g}_N$ if $N \ge 2$ and $\mathfrak{g}_1$ if $N=1$), on which it has a smooth density that is supported in the ball $\{ \parallel \lambda \parallel \le \rho \}$. Since the image by a submersion of a measure with a smooth and compactly supported density is a measure with a smooth density, we obtain the existence of a smooth function 
$q^N(x_0,\cdot):\mathbb{R}_{>0} \times
\mathbb{M} \rightarrow \mathbb{R}_{\ge 0}$, such that for
 every smooth $f :\mathbb{M} \rightarrow \mathbb{R} $,
\[
(\mathbf{Q}^N_t f)(x_0)=\int_\mathbb{M} q^N_t (x_0,y)f(y)dy.
\]

\

\textbf{Second step:}

\ If $f$ is a smooth function on $\mathbb{M}$, we denote by
$\mathcal{M}_f$ the operator on $\Gamma (\mathbb{M} ,
\mathcal{E})$ that acts by multiplication by $f$. Due to the
Leibniz rule for connections, we have for any word $I$:
\[
[\nabla_I, \mathcal{M}_f] =\mathcal{M}_{V_I f}.
\]
Consequently, if $\lambda=\sum_{I, d(I) \le N} \lambda_I X_I \in \bar{\mathfrak{g}}_N$,
\[
\left[ \sum_{I, d(I) \le N } \lambda_I \nabla_{I}
,\mathcal{M}_f \right]= \mathcal{M}_{\sum_{I, d(I) \le N }
\lambda_I V_{I} f}.
\]
The above commutation property implies the following one:
\[
\exp \left( \sum_{I, d(I) \le N } \lambda_I \nabla_{I}
\right) \mathcal{M}_f=\mathcal{M}_{\exp \left( \sum_{I, d(I) \le N
} \lambda_I V_{I} \right) f}\exp \left( \sum_{I, d(I) \le N
} \lambda_I \nabla_{I} \right).
\]
By applying this equality at the point $x_0$, we obtain that for every $\eta \in \Gamma (\mathbb{M} , \mathcal{E})$,
\[
\left[ \exp \left( \sum_{I, d(I) \le N } \lambda_I \nabla_{I}
\right) (f \eta) \right](x_0)=f(\Phi (\lambda)) \left[ \exp \left( \sum_{I, d(I) \le N } \lambda_I \nabla_{I}
\right) \eta \right](x_0).
\]
This implies that $\left[ \exp \left( \sum_{I, d(I) \le N } \lambda_I \nabla_{I}
\right) \eta \right](x_0)$ only depends on the value of $\eta$ at the point $\Phi (\lambda)$. In other words, there exists a linear map $\Theta (\lambda)$ from the fiber of $\mathcal{E}$ at $\Phi (\lambda)$ into the fiber of $\mathcal{E}$ at $x_0$ such that for every $\eta \in \Gamma (\mathbb{M} , \mathcal{E})$,
\begin{equation}\label{transport}
\left[ \exp \left( \sum_{I, d(I) \le N } \lambda_I \nabla_{I}
\right) \eta \right](x_0)=\Theta (\lambda) \eta ( \Phi (\lambda)).
\end{equation}
Since each operator $\nabla_I-\nabla_{V_I}$ is of order zero, that is a smooth section of $\mathbf{End}(\mathcal{E})$, the map $\Theta (\lambda)$ is obtained by solving a linear ordinary differential equation along the curve $s \rightarrow \exp \left( s \sum_{I, d(I) \le N } \lambda_I V_{I} \right) (x_0)$, $0 \le s \le 1$: It is a parallel transport along this curve, modified by terms of order zero. Therefore, $\Theta (\lambda)$ depends smoothly on $\lambda$. We deduce that for $\eta \in \Gamma (\mathbb{M} , \mathcal{E})$,
\[
(\mathbf{P}^N_t \eta)(x_0)=\mathbb{E} \left( \chi^N_t \Theta ( \Lambda ) \eta ( \Phi (\Lambda )) \right), \quad \text{where} \quad \Lambda=\sum_{I, d(I) \le N } \Lambda_I (B)_t X_{I}.
\]
We can now argue exactly as in the first step: The image by the submersion $\Phi$ of a measure with a smooth and compactly supported density, which is here multiplied by the smooth function $\Theta$, has a smooth density. We conclude that there exists a smooth map $y \rightarrow p^N_t (x_0,y)$, where $p^N_t (x_0,y)$ is a linear map from the fiber of $\mathcal{E}$ at $y$ into the fiber of $\mathcal{E}$ at $x_0$, such that for every $\eta \in \Gamma (\mathbb{M}, \mathcal{E} ) $,
\[
(\mathbf{P}^N_t \eta)(x_0)=\int_\mathbb{M} p^N_t (x_0,y)\eta(y)dy.
\]

\

\textbf{Final step:} \

Let us now turn to the proof of the pointwise estimates. In the sequel, we denote $\varepsilon=\sqrt{t}$. The idea is to show that both $\varepsilon^d p_{\varepsilon^2} (x_0,x_0)$ and $\varepsilon^d p^N_{\varepsilon^2} (x_0,x_0)$ admit asymptotic developments in powers of $\varepsilon$, whose coefficients are computed by the same rule: From the stochastic Taylor expansion for the first one, and from its approximation by the Chen-Strichartz formula for the second one.

We work in a system of normal coordinates centered at $x_0$ and in a local trivialization of the bundle $\mathcal{E}$ around $x_0$. In these coordinates, $x_0=0$, the sections of $\mathcal{E}$ are $\mathbb{R}^m$-valued functions, and we can write
\[
\nabla_i=V_i +A_i, \quad 0 \le i \le d,
\]
where the $A_i$'s are smooth functions that take their values in the space of $m \times m$ matrices. Moreover, we may assume that the $V_i$'s and the $A_i$'s are defined on $\mathbb{R}^d$, are $C^\infty$ bounded, and that the vector fields $V_1,...,V_d$ are uniformly elliptic. Indeed, on one hand, if $\rho$ is small enough, $p^N_t(x_0,x_0)$ only depends on the values of the $V_i$'s and of the $A_i$'s in a neighborhood of $x_0$ that may be chosen as small as we wish. On the other hand, modifying the operator $\mathcal{L}$ outside of a neighborhood of $x_0$ only modifies $p_t(x_0,x_0)$ by a term\footnote{Indeed, if $\tilde{p}_t(x_0,y)$ denotes the kernel of the modified operator, the difference $w_t(y)=p_t(x_0,y)-\tilde{p}_t(x_0,y)$, extended by $0$ when $t \le 0$, solves in a neighborhood of $x_0$ the equation $\frac{\partial w}{\partial t}=\mathcal{L}^* w$, where $\mathcal{L}^*$ is the adjoint of $\mathcal{L}$ (there is no singularity at $t=0$ because both kernels converge to the Dirac distribution at $x_0$ when $t \to 0$). Since $\mathcal{L}$ is elliptic around $x_0$, the operator $\frac{\partial }{\partial t}-\mathcal{L}^*$ is hypoelliptic, and thus $w$ is a smooth function of $(t,y)$ around $(0,x_0)$. Since $w$ vanishes for $t<0$, all its derivatives with respect to $t$ vanish at $t=0$.} which is $O(t^k)$ for every $k \ge 0$.

\

\textit{Asymptotic development of $p_t(x_0,x_0)$.} For $\varepsilon >0$, let $(X^\varepsilon_s, U^\varepsilon_s)_{s \ge 0}$ be the solution of the stochastic differential equations
\[
X^\varepsilon_s= \varepsilon^2 \int_0^s V_0(X_u^{\varepsilon}) du +\varepsilon \sum_{i=1}^d \int_0^s V_i
(X_u^{\varepsilon}) \circ dB^i_u, 
\]
\[
U^\varepsilon_s= \mathbf{Id}+ \varepsilon^2 \int_0^s U_u^{\varepsilon}A_0(X_u^{\varepsilon}) du +\varepsilon \sum_{i=1}^d \int_0^s U_u^{\varepsilon}A_i
(X_u^{\varepsilon}) \circ dB^i_u.
\]
It\^o's formula and the scaling property of Brownian motion show that for $\eta \in \Gamma (\mathbb{M}, \mathcal{E} )$, we have $(\mathbf{P}_{\varepsilon^2} \eta)(x_0)=\mathbb{E} \left( U^\varepsilon_1 \eta (X^\varepsilon_1) \right)$. Therefore,
\begin{equation}\label{representation1}
\varepsilon^d p_{\varepsilon^2} (x_0,x_0)=\mathbb{E} \left( U^\varepsilon_1 \delta_0 \left( \frac{X^\varepsilon_1}{\varepsilon} \right) \right),
\end{equation}
where the right hand side is a notation for the density at $0$ of the measure
\[
A \rightarrow \mathbb{E} \left( U^\varepsilon_1 1_A \left( \frac{X^\varepsilon_1}{\varepsilon} \right) \right).
\]
Now, exactly as in (\ref{Taylorexpansion}), for every $n \ge 0$, every smooth function $f$ and every section $\eta$, we have the stochastic Taylor expansions
\begin{equation}\label{Taylor1}
\begin{aligned}
f(X^\varepsilon_1)&=f (x_0)+\sum_{I, d(I) \le n} \varepsilon^{d(I)} \int_{\Delta^{\mid I \mid} [0,1]}  \circ dB^I (V_{i_1}
...V_{i_{\mid I \mid}}f)(x_0)+O(\varepsilon^{n+1}), \\
U^\varepsilon_1 \eta (X^\varepsilon_1)&=\eta (x_0)+\sum_{I, d(I) \le n} \varepsilon^{d(I)} \int_{\Delta^{\mid I \mid} [0,1]}  \circ dB^I (\nabla_{i_1}
...\nabla_{i_{\mid I \mid}}\eta)(x_0)+O(\varepsilon^{n+1}).
\end{aligned}
\end{equation}
By applying these expansions to the coordinate functions and to the constant sections, we deduce that
\[
\frac{X^\varepsilon_1}{\varepsilon}=\sum_{k=0}^n \varepsilon^k f_k +O(\varepsilon^{n+1}), \quad \quad U^\varepsilon_1=\sum_{k=0}^n \varepsilon^k u_k +O(\varepsilon^{n+1}),
\]
where $f_k$ (resp. $u_k$) is a linear combination of the iterated integrals $\int_{\Delta^{\mid I \mid} [0,1]}  \circ dB^I$, $d(I)=k+1$ (resp. $d(I)=k$), whose coefficients are given by (\ref{Taylor1}). For instance,
\[
f_0=\sum_{i=1}^d B^i_1 V_i(x_0), \quad \quad u_0=\mathbf{Id}.
\]
In these expansions, the remainder terms are $O(\varepsilon^{n+1})$ in all the Sobolev spaces of the Malliavin calculus. Moreover, since the vector fields $V_1,...,V_d$ are uniformly elliptic, the random variable $\frac{X^\varepsilon_1}{\varepsilon}$ is non degenerate in the sense of Malliavin, uniformly in $\varepsilon$. Under these conditions, a fundamental theorem of Watanabe (see \cite{Wat}, and Chapter V of \cite{Ike-Wat}) asserts that the above expansions may be composed with the Dirac distribution in (\ref{representation1}), exactly as in Lemma \ref{finite dimensional}: For every $n \ge 0$,
\begin{equation}\label{devpt}
\varepsilon^d p_{\varepsilon^2} (x_0,x_0)=\sum_{k=0}^n c_k \varepsilon^k +O(\varepsilon^{n+1}),
\end{equation}
where $c_k$ is the coefficient of $\varepsilon^k$ in the formal development (\ref{formal}), in which $g_l$ is replaced by $u_l$ and where the expectations are understood as in (\ref{convention}). Since $(V_1(x_0),...,V_d(x_0))$ is orthonormal, the density of the Gaussian random variable $f_0$ is here $g(y)=\frac{1}{(2\pi)^{d/2}} e^{-\frac{\| y \|^2}{2}}$. For instance, $c_0=\frac{1}{(2\pi)^{d/2}} \mathbf{Id}$. For our purpose, the two important points are the following. First, $c_k$ only depends on $f_0,...,f_k$ and $u_0,...,u_k$. Secondly,
\begin{equation}\label{parity}
c_k=0, \quad \text{if } k \text{ is odd}.
\end{equation}
Indeed, if we change the Brownian motion $(B_s)_{s \ge 0}$ into the Brownian motion $(-B_s)_{s \ge 0}$, then an iterated integral $\int_{\Delta^{\mid I \mid} [0,1]}  \circ dB^I$ is multiplied by $(-1)^{d(I)}$, so that $f_l$ is changed into $(-1)^{l+1} f_l$ and $u_l$ into $(-1)^l u_l$. Since the derivatives of order $m$ of the Dirac distribution at $0$ are multiplied by $(-1)^m$ by the symmetry $y \rightarrow -y$, the coefficient of $\varepsilon^k$ in (\ref{formal}) is therefore multiplied by $(-1)^k$. But this coefficient does not depend on the choice of the Brownian motion, and thus $c_k=(-1)^k c_k$.

\

\textit{Asymptotic development of $p^N_t(x_0,x_0)$.} In our coordinates, the map $\Phi$ of the first step takes its values in $\mathbb{R}^d$, and the map $\Theta$ of the second step takes its values in the space of $m \times m$ matrices. With the notations of Lemma \ref{finite dimensional}, we denote
\[
\tilde{X}^\varepsilon=\Phi (\Lambda^\varepsilon), \quad \quad \tilde{U}^\varepsilon=\Theta (\Lambda^\varepsilon),
\]
so that, from (\ref{transport}), for every $\eta \in \Gamma (\mathbb{M}, \mathcal{E} )$,
\[
\left[ \exp \left( \sum_{I, d(I) \le N } \varepsilon^{d(I)} \Lambda_I (B)_1 \nabla_{I}
\right) \eta \right](x_0) =\tilde{U}^\varepsilon \eta ( \tilde{X}^\varepsilon ).
\]
From the scaling property (\ref{scaling}), we have then
\begin{equation}\label{representation2}
\varepsilon^d p^N_{\varepsilon^2} (x_0,x_0)=\mathbb{E} \left( \tilde{\chi}^\varepsilon \tilde{U}^\varepsilon \delta_0 \left( \frac{\tilde{X}^\varepsilon}{\varepsilon} \right) \right).
\end{equation}
As we have seen in the proof of Theorem \ref{Parametrix}, the Chen-Strichartz formula implies that the expansions (\ref{Taylor1}) also hold, when $n=N$, for $\tilde{X}^\varepsilon$ and $\tilde{U}^\varepsilon$. In other words, the Taylor expansions with respect to $\varepsilon$ of $\frac{\tilde{X}^\varepsilon}{\varepsilon}$ and $\tilde{U}^\varepsilon$ are $\sum_{k \ge 0} \varepsilon^k \tilde{f}_k$ and $\sum_{k \ge 0} \varepsilon^k \tilde{u}_k$ with
\begin{equation}\label{coincidence}
\tilde{f}_k=f_k, \quad 0 \le k \le N-1, \quad \quad \tilde{u}_k=u_k, \quad 0 \le k \le N.
\end{equation}
Since the differential of $\Phi$ at $0$ maps $X_i$ onto $V_i(x_0)$, we can apply Lemma \ref{finite dimensional} to the maps $\Phi$ and $G=\Theta$: If $\rho$ is small enough, the right hand side of (\ref{representation2}) is a smooth function of $\varepsilon$ up to $\varepsilon=0$, whose Taylor expansion $\sum_{k \ge 0} \tilde{c}_k \varepsilon^k$ is given by the formal rule (\ref{formal}), with $\tilde{f}_l$ and $\tilde{u}_l$ in place of $f_l$ and $g_l$.

\

\textit{Conclusion.} Since $c_k$ and $\tilde{c}_k$ are given by the same expression in terms of $(f_l,u_l)_{0 \le l \le k}$ and $(\tilde{f}_l,\tilde{u}_l)_{0 \le l \le k}$ respectively, we deduce from (\ref{coincidence}) that $c_k=\tilde{c}_k$ for $0 \le k \le N-1$. As a consequence,
\[
\varepsilon^d \left( p_{\varepsilon^2} (x_0,x_0)-p^N_{\varepsilon^2} (x_0,x_0) \right)=O\left(\varepsilon^N\right),
\]
which is the first estimate we had to prove. Finally, the argument which gave (\ref{parity}) also gives $\tilde{c}_k=0$ if $k$ is odd, because $\Lambda_I(B)_1$ is multiplied by $(-1)^{d(I)}$ when $B$ is changed into $-B$. Therefore, if $N$ is odd, $c_N=\tilde{c}_N=0$ and
\[
\varepsilon^d \left( p_{\varepsilon^2} (x_0,x_0)-p^N_{\varepsilon^2} (x_0,x_0) \right)=O\left(\varepsilon^{N+1}\right).
\]
\end{proof}

\begin{remark}\label{remark sphere}
Thanks to the cutoff, $p_t^N (x_0,\cdot)$ is smooth. Without the cutoff, the first step of the previous proof may fail. Consider for instance on the sphere $\mathbb{S}^2 \subset \mathbb{R}^3$ the vector fields $V_1(x)=e_1 \times x$ and $V_2(x)=e_2 \times x$, where $(e_1,e_2,e_3)$ is the canonical basis of $\mathbb{R}^3$, and let $x_0=e_3$. If $b \in \mathbb{R}^2$, then $\exp (b_1 V_1+b_2V_2)(x_0)$ is the image of $x_0$ by the rotation with axis $b_1e_1+b_2e_2$ and angle $\| b \|$. In particular, all the circles $\| b \|=2k\pi$, $k \ge 1$, are sent to the point $x_0$, and it is easily checked that, as a consequence, the density of the random variable $\exp (B^1_t V_1+B^2_tV_2)(x_0)$ at a point $y$ is equivalent to $\frac{c_t}{d(x_0,y)}$, $c_t>0$, when $y \to x_0$. It is therefore not possible to speak of the value of this density at $x_0$.
\end{remark}

From the previous theorem, we deduce an explicit asymptotic expansion of $p_t (x_0,x_0)$.
If $I\in \{0,1,...,d\}^k$, $k \ge 2$, is a word, we denote
\[
\mathcal{F}_I=\nabla_I-\nabla_{V_I} \in \Gamma
(\mathbb{M},\mathbf{End}(\mathcal{E})).
\]

\begin{corollary}\label{coro}
For $N \ge 1$, let us denote
\[
W^N_t=\sum_{I, d(I) \le N } \Lambda_I (B)_t V_{I}, \quad \quad J^N_t=\left| \det \left( \int_0^1 e^{s DW^N_t (x_0)} ds \right) \right|^{-1},
\]
where $DW^N_t (x_0)$ is the endomorphism $v \rightarrow D_v W^N_t$ of the tangent space at $x_0$, $D$ being the Levi-Civita connection of $\mathbb{M}$. If the parameter $\rho$ of the cutoff is small enough, then, when $t \rightarrow 0$,
\begin{align*}
p_t (x_0,x_0)= & d^N_t (x_0)\mathbb{E} \left( \chi^N_t J^N_t \exp \left( \sum_{I,
d(I) \le N } \Lambda_I (B)_t \mathcal{F}_{I}(x_0) \right) \left|
\sum_{I, d(I) \le N } \Lambda_I (B)_t V_{I}(x_0)=0\right) \right. \\
 & + O\left( t^{\frac{N-d}{2}}\right),
\end{align*}
where $d^N_t(x_0)$ is the density at $0$ of the random variable $
\sum_{I, d(I) \le N } \Lambda_I (B)_t V_{I}(x_0)$. Moreover, if $N$ is odd, then the remainder term is $O\left( t^{\frac{N+1-d}{2}}\right)$.
\end{corollary}

\begin{proof}
Let us first observe that, since $(V_1(x_0),...,V_d(x_0))$ is a basis of the tangent space at $x_0$, the random variable
\[
\sum_{I, d(I) \le N } \Lambda_I (B)_t V_{I}(x_0)
\]
is the image of the random variable of Proposition \ref{density} by a surjective affine map. It has therefore a smooth density $d^N_t (x_0,\cdot)$, and the conditional expectation of the statement is well defined\footnote{More precisely, the product of $d^N_t(x_0)$ by this conditional expectation is, by definition, the density at $0$ of the measure $A \rightarrow \mathbb{E} \left( \chi^N_t J^N_t \exp \left( \sum_{I,
d(I) \le N } \Lambda_I (B)_t \mathcal{F}_{I}(x_0) \right) 1_A \left( \sum_{I, d(I) \le N } \Lambda_I (B)_t V_{I}(x_0) \right) \right)$, which is smooth for the same reason.}.

The idea is now to compare the random variables $Z_t^N=\exp ( W^N_t )(x_0)$ and $W^N_t(x_0)$. Let $W=\sum_{I, d(I) \le N } \lambda_I V_{I}$ be a vector field such that $\parallel \sum_{I, d(I) \le N } \lambda_I X_{I} \parallel \le \rho$ and let us denote $\phi_s=\exp (sW)$. We work in a system of normal coordinates centered at $x_0$ and denote by $D \phi_s$ the Jacobian matrix of $\phi_s$. Since the flow of $W$ leaves $W$ invariant, we have for $0 \le s \le 1$,
\[
\frac{d}{ds} \phi_s (x_0)=W( \phi_s (x_0))=D\phi_s (x_0) W(x_0).
\]
By integrating with respect to $s$, we obtain
\begin{equation}\label{linearization}
\exp (W)(x_0)-x_0=\left( \int_0^1 D\phi_s (x_0) ds \right) W(x_0).
\end{equation}
If $\rho$ is small enough, the matrix $\int_0^1 D\phi_s (x_0) ds$ is close to the identity matrix, and therefore invertible. As a consequence, $\exp (W)(x_0)=x_0$ if and only if $W(x_0)=0$. Moreover, in that case:
\begin{itemize}
\item The point $x_0$ is a fixed point of $\phi_s$, so that $D\phi_s (x_0)=e^{s DW(x_0)}$. The inverse of the absolute value of the determinant of the matrix $\int_0^1 D\phi_s (x_0) ds$ is thus the value of $J^N_t$ when $W^N_t=W$;
\item Since $\nabla_I=\mathcal{F}_I+\nabla_{V_I}$, the solution of the equation $\frac{\partial \Psi}{\partial \tau}=\sum_{I, d(I) \le N }
\lambda_I \nabla_{I} \Psi$, $\Psi(0,\cdot)=\eta$, satisfies $\frac{\partial \Psi}{\partial \tau}(\tau,x_0)=\sum_{I, d(I) \le N }
\lambda_I \mathcal{F}_{I}(x_0) \Psi(\tau,x_0)$, and therefore
\[
\left[ \exp \left( \sum_{I, d(I) \le N } \lambda_I \nabla_{I} \right) \eta \right] (x_0)=\exp \left( \sum_{I, d(I) \le N } \lambda_I \mathcal{F}_{I} (x_0) \right) \eta (x_0), \quad \eta \in \Gamma (\mathbb{M}, \mathcal{E} ).
\]
\end{itemize}
We can now compute $p^N_t(x_0,x_0)$ by using the change of variable (\ref{linearization}): On the event $\chi^N_t \neq 0$, the random variable $Z^N_t -x_0$ is the image of $W^N_t(x_0)$ by an invertible matrix that depends smoothly on $\sum_{I, d(I) \le N } \Lambda_I (B)_t X_{I}$. Therefore, the density at $x_0$ of $Z_t^N$ is obtained from the density at $0$ of $W_t^N(x_0)$ by dividing by the absolute value of the determinant of this matrix, computed on the event $W_t^N(x_0)=0$. From the two previous observations, we thus get that $p^N_t (x_0,x_0)$ is equal to
\[
d^N_t (x_0)\mathbb{E} \left( \chi^N_t J^N_t \exp \left( \sum_{I,
d(I) \le N } \Lambda_I (B)_t \mathcal{F}_{I}(x_0) \right) \left|
\sum_{I, d(I) \le N } \Lambda_I (B)_t V_{I}(x_0)=0\right) \right.,
\]
and the proof follows from Theorem \ref{main-estimation}.
\end{proof}

The previous theorem may be used to prove various local index theorems. We focus here on the simplest one which is the Chern-Gauss-Bonnet theorem. A proof of the Atiyah-Singer local index theorem for the Dirac operator on spin manifolds can be found in \cite{Bau3}.

For alternative proofs that use stochastic differential geometry, we refer to \cite{Bi2} and \cite{Hsu}.
\subsection{The Chern-Gauss-Bonnet theorem}

Let $\mathbb{M}$ be a $d$-dimensional closed Riemannian manifold, that is a smooth and compact
Riemannian manifold without boundary. The  Chern-Gauss-Bonnet theorem proved by Chern
\cite{Chern} in 1944\footnote{Chern's paper deals with orientable manifolds. The statement below does not require an orientation: The non orientable case may be deduced from the orientable one by considering the orientation covering of $\mathbb{M}$, and it is also covered by the proof which is given in this section.} is the following:

\begin{theorem}
Let $\chi (\mathbb{M})$ be the Euler characteristic of
$\mathbb{M}$. If $d$ is odd, then $\chi (\mathbb{M})=0$. If $d$ is
even then
\[
\chi (\mathbb{M}) =\int_\mathbb{M} \omega (x)dx,
\]
where $dx$ is the Riemannian measure of $\mathbb{M}$ and where $\omega(x)dx$ is the Euler density: The function $\omega$ is given in
a local orthonormal frame $e_i$ by
\[
\omega=\frac{(-1)^{d/2}}{(8\pi)^{d/2} (d/2)! } \sum_{\sigma,\tau
\in \Sigma_d} \epsilon (\sigma) \epsilon (\tau) \prod_{i \in \{1,3,...,d-1 \}}
R_{\sigma (i) \sigma (i+1) \tau (i) \tau (i+1) },
\]
where  $\Sigma_d$ is the set of the permutations of the indices
$\{1,...,d\}$, $\epsilon$ the signature of a permutation, and
\[
R_{ijkl}=\left\langle R(e_i,e_j)e_k,e_l \right\rangle,
\]
with $R(X,Y)=\nabla_X \nabla_Y-\nabla_Y \nabla_X -\nabla_{[X,Y]}$ Riemannian curvature of $\mathbb{M}$. This expression does not depend on the choice of the orthonormal frame, so that $\omega$ is a globally defined function on $\mathbb{M}$, even if $\mathbb{M}$ is not orientable.
\end{theorem}
The striking feature of Chern-Gauss-Bonnet theorem that makes it
so beautiful is that the Euler density depends on the Riemannian
metric whereas  $\chi (\mathbb{M}) $ is only a topological
invariant. We now turn to a short proof of it that uses the tools we developed in these notes. In the sequel, we always assume that the dimension $d$ is even.

\

We first briefly recall some basic facts on
Fermion calculus on the Clifford exterior algebra of a finite
dimensional vector space, as can be found in Section 2.2.2 of
\cite{Ros} and that will be used in our proof. Let $V$ be a
$d$-dimensional Euclidean vector space. We denote $V^\ast$ its
dual and
\[
\wedge V^\ast=\bigoplus_{k \ge 0} \wedge^k V^\ast,
\]
the exterior algebra. If $u \in V^\ast$, we denote $a^\ast_u$ the
map $\wedge V^\ast \rightarrow \wedge V^\ast$, such that $a^\ast_u
(\omega)=u \wedge \omega$. The dual map is denoted $a_u$. Let now
$\theta_1$, ..., $\theta_d$ be an orthonormal basis of $V^\ast$.
We denote $a_i=a_{\theta_i}$. We have the basic rules of Fermion
calculus
\[
\{ a_i ,a_j \}=0, \{ a^\ast_i ,a^\ast_j \}=0, \{ a_i ,a^\ast_j
\}=\delta_{ij},
\]
where $\{ \cdot , \cdot \}$ stands for the anti-commutator: $\{
a_i ,a_j \}=a_i a_j +a_j a_i$. If $I$ and $J$ are two words with
$1 \le i_1 < \cdots < i_k \le d$ and $1 \le j_1 < \cdots < j_l \le
d$, we denote
\[
A_{IJ}=a^\ast_{i_1} \cdots a^\ast_{i_k}   a_{j_1} \cdots a_{j_l}.
\]
The family of all the possible $A_{IJ}$ forms a basis of the
$2^{2d}$-dimensional vector space $\mathbf{End} \left( \wedge
V^\ast\right)$.  

 If $A = \sum_{I,J} c_{IJ} A_{IJ} \in \mathbf{End} \left( \wedge V^\ast\right)$, we shall say that $A$ is of degree at most $m$ if $c_{IJ}=0$ as soon as the word $I$ has more than $m$ letters. It is easily seen from the rules of Fermion calculus that if $A$ is of degree at most $m$ and $A'$ of degree at most $m'$, then $AA'$ is of degree at most $m+m'$.

If $A \in \mathbf{End} \left( \wedge
V^\ast\right)$, we define the supertrace $\mathbf{Str} (A)$ as the
difference of the trace of $A$ on even forms minus the trace of
$A$ on odd forms. 

One of the the interests of Fermion calculus (which is
equivalent to Clifford calculus) is that it makes easy to compute  supertraces: If $A = \sum_{I,J} c_{IJ} A_{IJ}$,
then
\begin{align}\label{supertraceformula}
\mathbf{Str} (A)=(-1)^{\frac{d(d-1)}{2}} c_{\{1,...,d\}\{1,...,d\}}.
\end{align}
In particular, $\mathbf{Str} (A)=0$ as soon as $A$ is of degree at most $d-1$.

We now carry the Fermionic construction on the tangent spaces of
our manifold $\mathbb{M}$. Let $e_i$ be a local orthonormal frame
and let $\theta_i$ be its dual frame. The curvature endomorphism
is defined by
\[
\mathcal{F}=\sum_{ijkl} R_{ijkl}a_i^\ast a_j a_k^\ast a_l
\]
where
\[
R_{ijkl}=\left\langle R(e_i,e_j)e_k,e_l \right\rangle,
\]
with $R$ Riemannian curvature of $\mathbb{M}$. This definition is
actually intrinsic, i.e. does not depend on the choice of the
local frame. In this setting, the celebrated Weitzenb\"ock formula
reads
\[
\square=\Delta + \mathcal{F},
\]
where $\square =d d^\ast +d^\ast d$ is the Hodge-DeRham Laplacian
and $\Delta$ the Bochner Laplacian. Let us recall that if $e_i$ is
a local orthonormal frame, we have the following explicit formula
for $\Delta$:
\[
\Delta=-\sum_{i=1}^d (\nabla_{e_i}\nabla_{e_i}- \nabla_{
\nabla_{e_i} e_i }),
\]
where $\nabla$ is the Levi-Civita connection.

\

After these preliminaries, we can now turn to the proof of
Chern-Gauss-Bonnet theorem. From now on, we suppose that the
dimension $d$ is even. The first crucial step is McKean-Singer
formula \cite{McK-Sin} (A simple proof of it can be found in \cite{Ros} page 113). We have
\[
\chi(\mathbb{M})=\int_{\mathbb{M}}\mathbf{Str}\text{ } p_t (x,x) dx
, \quad t >0,
\]
where $\mathbf{P}_t=e^{-t \square}$ and $p_t$ is the corresponding
Schwartz kernel (density). In other words, the supertrace is constant
along the heat semigroup associated with the Hodge-DeRham
 Laplacian and this constant is equal to the Euler characteristic. 
 
 An easily proved and non surprising
precise statement is the following: When $t \rightarrow 0$,
\[
\sup_{x \in \mathbb{M}} \parallel p_t (x,x) - \frac{1}{(4 \pi
t)^{\frac{d}{2}}}e^{-t \mathcal{F}}(x) \parallel =O\left(
\frac{1}{t^{d/2 -1}}\right).
\]
But as seen in the next proposition, due to Corollary \ref{coro}, \textit{fantastic and subtle
cancellations}\footnote{We quote here McKean-Singer \cite{McK-Sin}
who conjectured these cancellations of terms that involve many
covariant derivatives of curvature terms.} occur at the paths
level when we take the supertrace:

\begin{proposition}\label{local-Chern}
For every $x \in \mathbb{M}$, \[
\lim_{t \rightarrow 0}\mathbf{Str}\text{ } p_t (x,x)=\omega (x).
\]
\end{proposition}

\begin{proof}
Let $x_0 \in \mathbb{M}$ be fixed once time or all in the following proof. We work in a synchronous local orthonormal frame $e_i$ around $x_0$, that is $\nabla e_i =0$ at $x_0$.
In a neighborhood of $x_0$, we have
\[
\Delta=-\sum_{i=1}^d \nabla_{e_i}\nabla_{e_i} -2 \nabla_{e_0}, \quad \text{where} \quad e_0=-\frac{1}{2} \sum_{i=1}^d \nabla_{e_i} e_i,
\] 
and therefore
\[
-\frac{1}{2} \square=\frac{1}{2} \sum_{i=1}^d \nabla_{e_i}\nabla_{e_i} + \nabla_{e_0} -\frac{1}{2} \mathcal{F}.
\]
Let us observe that the vector field $e_0$ vanishes at $x_0$, but that it has to be taken into account: It does not vanish in a neighborhood of $x_0$, and we shall need to consider brackets of the above operators.
We want to apply Corollary \ref{coro} in the present framework\footnote{The frame $(e_i)$ is only defined around $x_0$, whereas Corollary \ref{coro} was stated for globally defined vector fields. This is not a problem since, as we have seen in the final step of the proof of Theorem \ref{main-estimation}, the kernels at $(x_0,x_0)$ of two operators that coincide in a neighborhood of $x_0$ only differ by a term which is $O(t^k)$ for every $k \ge 0$.}. 

We denote $\mathcal{F}_0=-\frac{1}{2} \mathcal{F}$, $\mathcal{F}_i=0$, $1\le i \le d$ and if $I\in \{0,1,...,d\}^k$ is a word, 
\[
\mathcal{F}_I= [\nabla_{i_1},[\nabla_{i_2},...,[\nabla_{i_{k-1}},
\nabla_{i_{k}}]...]-\nabla_{ [e_{i_1},[e_{i_2},...,[e_{i_{k-1}},
e_{i_{k}}]...]} \in \Gamma
(\mathbb{M}, \mathbf{End} \left( \wedge
T^\ast \mathbb{M} \right)),
\]
with the convention $\nabla_0=\nabla_{e_0}+\mathcal{F}_0$, $\nabla_i =\nabla_{e_i}$, $1\le i \le d$. We also denote $e_I=[e_{i_1},[e_{i_2},...,[e_{i_{k-1}},
e_{i_{k}}]...]$, so that $\nabla_I=\nabla_{e_I}+\mathcal{F}_I$.

The algebraic key of the proof is the following observation: For every word $I \neq (0)$,
\begin{equation}\label{degree}
\mathcal{F}_I \text{ is of degree at most } d(I)-1,
\end{equation}
whereas $\mathcal{F}_0$ is of degree at most $2$. Let us prove (\ref{degree}) by induction on the length of the word $I$. If $I=(i)$, $1 \le i \le d$, then $\mathcal{F}_I=0$. Let now $I$ be a word, $0 \le i \le d$, and let us denote by $(i,I)$ the word that is obtained by adding the letter $i$ in front of $I$. If $X,Y$ are vector fields and $A,B$ sections of $\mathbf{End} \left( \wedge
T^\ast \mathbb{M} \right)$, we have
\[
[\nabla_X+A,\nabla_Y+B]=\nabla_{[X,Y]}+\mathcal{R}(X,Y)+(\nabla_X B)-(\nabla_Y A)+[A,B],
\]
where $\mathcal{R}(X,Y)=\nabla_X \nabla_Y-\nabla_Y \nabla_X -\nabla_{[X,Y]}$ is the curvature of the connection $\nabla$ on $\wedge
T^\ast \mathbb{M}$, and where $(\nabla_X B)$ is the covariant derivative of $B$. Therefore,
\[
\mathcal{F}_{(i,I)}=\mathcal{R}(e_i,e_I)+(\nabla_{e_i} \mathcal{F}_I)-(\nabla_{e_I} \mathcal{F}_i)+[\mathcal{F}_i,\mathcal{F}_I].
\]
Now, $\mathcal{R}(X,Y)=-\sum_{k,l} \left\langle R(X,Y)e_k,e_l \right\rangle a_k^\ast a_l$ is of degree at most $1$, $\mathcal{F}_0$ is of degree at most 2, and the covariant derivative does not increase the degree. Consequently:
\begin{itemize}
\item If $1 \le i \le d$, then $\mathcal{F}_i=0$, so that $\mathcal{F}_{(i,I)}$ is of degree at most $\max (1,m)$ as soon as $\mathcal{F}_I$ is of degree at most $m$. We may take $m=d(I)$ (by induction if $I \neq (0)$, and because $\mathcal{F}_0$ is of degree at most $2$ if $I=(0)$). Hence, $\mathcal{F}_{(i,I)}$ is of degree at most $d(I)=d((i,I))-1$;
\item If $i=0$ and $I \neq (0)$, then $\mathcal{F}_{(0,I)}$ is of degree at most $m+2$ as soon as $\mathcal{F}_I$ is of degree at most $m$. By induction, we may take $m=d(I)-1$. Hence, $\mathcal{F}_{(0,I)}$ is of degree at most $d(I)+1=d((0,I))-1$;
\item Finally, $\mathcal{F}_{(0,0)}=0$.
\end{itemize}

Since $d$ is even, we can not apply Corollary \ref{coro} with $N=d$, for which the remainder term would only be $O(1)$. We therefore apply it with $N=d+1$, which is odd. In the sequel of the proof, we denote
\[
\xi_t=\sum_{I, d(I) \le d+1 } \Lambda_I (B)_t e_{I}(x_0),
\]
and by $q_t (x_0)$ the density at $0$ of the random variable $\xi_t$. According to Corollary \ref{coro}, we thus have when $t \rightarrow 0$,
\[
p_{t/2} (x_0,x_0)=q_t (x_0)\mathbb{E} \left( \chi^{d+1}_t J^{d+1}_t \exp \left( \sum_{I,
d(I) \le d+1 } \Lambda_I (B)_t \mathcal{F}_{I}(x_0) \right) \ \middle| \ \xi_t=0\right) 
+ O\left( t \right).
\]

By using the scaling property (\ref{scaling}) of Brownian motion and Lemma \ref{finite dimensional}, we see that if $I_1,...,I_k$ are words, then
\begin{equation}\label{scaling2}
q_t (x_0)\mathbb{E} \left( \chi^{d+1}_t \left| \Lambda_{I_1} (B)_t \cdots \Lambda_{I_k} (B)_t \right| \ \middle| \ \xi_t=0\right)=O\left( t^{\frac{d(I_1)+\cdots+d(I_k)-d}{2}} \right),
\end{equation}
and that
\begin{equation}\label{scaling3}
\lim_{t \to 0} t^{\frac{d}{2}} q_t (x_0)=\frac{1}{(2\pi)^{\frac{d}{2}}}, \quad \quad \mathbb{E} \left( \chi^{d+1}_t  J^{d+1}_t \ \middle| \ \xi_t=0\right)=1+O\left(\sqrt{t}\right).
\end{equation}
On the event $\chi^{d+1}_t \neq 0$, the random variables $J^{d+1}_t$ and $\sum_{I,
d(I) \le d+1 } \Lambda_I (B)_t \mathcal{F}_{I}(x_0)$ are bounded by deterministic constants. Since $\mathcal{F}_{I}(x_0)=0$ if $d(I)=1$, we deduce from (\ref{scaling2}) that the terms of order $k \ge \frac{d}{2}+1$ in the development of the exponential only contribute to the remainder term. Therefore
\[
p_{t/2} (x_0,x_0)=q_t (x_0)\mathbb{E} \left( \chi^{d+1}_t J^{d+1}_t \sum_{k=0}^{d/2} \frac{1}{k!}  \left( \sum_{I,
d(I) \le d+1 } \Lambda_I (B)_t \mathcal{F}_{I}(x_0) \right)^k \ \middle| \ \xi_t=0\right) 
+ O\left( t \right).
\]
Let us now take the supertrace. We have
\[
\mathbf{Str}\text{ }\left( \sum_{I,
d(I) \le d+1 } \Lambda_I (B)_t \mathcal{F}_{I}(x_0) \right)^k =\sum_{I_1,...,I_k} \Lambda_{I_1} (B)_t \cdots \Lambda_{I_k} (B)_t \text{ } \mathbf{Str}\text{ } \left( \mathcal{F}_{I_1}(x_0) \cdots \mathcal{F}_{I_k}(x_0) \right).
\]
From (\ref{degree}), the product $\mathcal{F}_{I_1}(x_0) \cdots \mathcal{F}_{I_k}(x_0)$ is of degree at most $d(I_1)+\cdots+d(I_k)-l$, where $l$ is the number of words among $I_1,...,I_k$ that are different from $(0)$. Due to the formula (\ref{supertraceformula}), its supertrace is therefore zero unless $d(I_1)+\cdots+d(I_k) \ge d+l$. On the other hand, from (\ref{scaling2}), the terms such that $d(I_1)+\cdots+d(I_k) \ge d+1$ only contribute to $O\left( \sqrt{t} \right)$. The only term that remains is thus obtained when $l=0$ and $d(I_1)+\cdots+d(I_k)=d$, that is when $k=\frac{d}{2}$ and $I_1=\cdots=I_k=(0)$. Since $\Lambda_{(0)} (B)_t=t$, we obtain 
\[
\mathbf{Str}\text{ }p_{t/2} (x_0,x_0)= \frac{1}{(d/2)!} t^{\frac{d}{2}} \text{ } \mathbf{Str}\text{ } \mathcal{F}_{0}(x_0)^{\frac{d}{2}} \text{ } q_t (x_0) \mathbb{E} \left( \chi^{d+1}_t J^{d+1}_t \ \middle| \ \xi_t=0\right) + O\left( \sqrt{t} \right).
\]
From (\ref{scaling3}), this proves that
 \[
 \lim_{t \to 0} \mathbf{Str}\text{ }p_{t} (x_0,x_0) =\frac{(-1)^{d/2} }{(d/2)! (4\pi)^{d/2}} \mathbf{Str}\text{ } \mathcal{F}(x_0)^{\frac{d}{2}}.
 \]
 Heavy, but straightforward computations (see \cite{Ros}, Lemma 2.35) show that
 \[
 \frac{(-1)^{d/2} }{(d/2)! (4\pi)^{d/2}} \mathbf{Str}\text{ } \mathcal{F}(x_0)^{\frac{d}{2}} =\omega(x_0).
 \]
\end{proof}

To conclude the proof of the Chern-Gauss-Bonnet theorem, we observe that, since $\mathbb{M}$ is compact, the asymptotic development of $p_t(x,x)$ in powers of $t$ is uniform with respect to $x \in \mathbb{M}$ (see Chapter 2 of \cite{Ber-Ge-Ve}). From Proposition \ref{local-Chern}, the convergence of $\mathbf{Str}\text{ } p_t (x,x)$ to $\omega(x)$ is therefore uniform on $\mathbb{M}$, and by letting $t \to 0$ in McKean-Singer formula we obtain
\[
\chi (\mathbb{M}) =\int_\mathbb{M} \omega (x)dx.
\]

\end{document}